\documentclass[11pt, reqno]{amsart}
\usepackage{amsmath, amssymb, amsthm, amsfonts, bm}
\usepackage{enumitem}
\usepackage[margin=1.3in]{geometry}
\usepackage[
colorlinks=true, urlcolor=blue, linkcolor=blue, citecolor=magenta]{hyperref}
\usepackage{amsrefs}
\allowdisplaybreaks
\usepackage{graphicx}
\usepackage{comment}

\newtheorem{theorem}{Theorem}[section]
\newtheorem{proposition}[theorem]{Proposition}

\newtheorem{lemma}[theorem]{Lemma}
\newtheorem{corollary}[theorem]{Corollary}

\theoremstyle{remark}
\newtheorem{remark}{Remark}[section]
\theoremstyle{definition}

\numberwithin{equation}{section}

\DeclareMathOperator{\tr}{tr}

\DeclareMathOperator{\diam}{diam}
\DeclareMathOperator{\dist}{dist}
\DeclareMathOperator{\Int}{Int}

\newcommand{\Id}{\operatorname{Id}}
\newcommand{\Sph}{\mathbb S}

\newcommand{\B}{\mathbb B}
\newcommand{\R}{\mathbb R}
\newcommand{\HH}{\mathbb H}
\newcommand{\PiH}{\Pi}
\newcommand{\cL}{\mathcal L}
\newcommand{\cE}{\mathcal E}

\newcommand{\dd}{\,d}
\newcommand{\bd}{\partial}
\newcommand{\ol}{\overline}

\newcommand{\ip}[2]{\left\langle #1,#2\right\rangle}

\begin{document}

\title{Free boundary flows by powers of the Gauss curvature in the unit ball}

\author[T. Luo]{Tianci Luo}
\author[Y. Wei]{Yong Wei}
\author[R. Zhou]{Rong Zhou}
\address{School of Mathematical Sciences, University of Science and Technology of China, Hefei 230026, P.R. China}
\email{Luo\_tianci@mail.ustc.edu.cn}
\email{yongwei@ustc.edu.cn}
\email{zhourong@mail.ustc.edu.cn}

\date{\today}

\begin{abstract}
We study smooth compact strictly convex hypersurfaces in the unit ball that meet the support sphere orthogonally and evolve by the $\alpha$-Gauss curvature flow $\partial_tX=-K^\alpha\nu$, $\alpha>0$.  We prove that the solution remains strictly convex,  becomes extinct in finite time and contracts to a single point on the support sphere. If $\alpha >\frac{1}{n+2}$, we apply a Cayley-type conformal map that sends the extinction point to the origin of a Euclidean
half-space and then normalize the enclosed half-space volume. The resulting normalized hypersurfaces converge smoothly to the unit hemisphere.  The proof combines boundary identities for the spherical free boundary, a boundary-adapted Tso estimate, an almost-monotonicity
formula for a half-space entropy, and uniform curvature estimates for the normalized flow.
\end{abstract}

\maketitle
\tableofcontents

\section{Introduction}
\label{sec:intro}

The Gauss curvature flow is a fundamental example of a fully nonlinear
curvature flow for convex hypersurfaces. In Euclidean space, the family
\begin{equation}\label{eq:closed-alpha-intro}
    \partial_tX=-K^\alpha\nu,
    \qquad
    \alpha>0,
\end{equation}
contains the classical Gauss curvature flow introduced by Firey
\cite{Firey1974} as a model for the abrasion of convex bodies. For
smooth, closed, uniformly convex hypersurfaces in
$\mathbb R^{n+1}$, Tso \cite{Tso1985} proved that the flow with
$\alpha=1$ remains uniformly convex, becomes extinct in finite time,
and contracts to a point. Chow \cite{chow85} treated the flow by the
$n$th root of the Gauss curvature. Since then, flows by powers of the
Gauss curvature have become a central class of fully nonlinear
curvature flows.

The asymptotic behavior near the extinction time has been studied extensively.  Andrews \cite{And99} proved that every smooth, closed, uniformly convex surface in $\mathbb{R}^3$ contracts to a round point under the Gauss curvature flow, thereby resolving Firey's conjecture.  The case $\alpha={1}/{(n+2)}$ is the affine normal flow.  For this flow, Andrews \cite{And96} proved that, after reparametrization and suitable rescaling about the final point, the solution converges to a normalized ellipsoid.  In the superaffine range $\alpha>1/(n+2)$, entropy methods give convergence of the volume-normalized closed flow to a smooth uniformly convex self-similar solution.  This was established by Andrews \cite{And00} for $\alpha\in (\frac{1}{n+2}, \frac{1}{n}]$, by Guan--Ni \cite{GN17} for $\alpha=1$, and by Andrews--Guan--Ni \cite{AGN2016} for all $\alpha>1/(n+2)$.  Combined with the classification theorem of Brendle--Choi--Daskalopoulos \cite{BCD2017} (see also earlier classification by Choi--Daskalopoulos \cite{CD16} for $\alpha\in [\frac{1}{n},1+\frac{1}{n})$), these results imply convergence to the round sphere.  The two-dimensional case for powers $\alpha\in [1/2,1]$ was treated in \cite{AC12}.  We refer the reader to \cite[Chapters 15--17]{Andbook} for a comprehensive exposition of smooth closed solutions to the $\alpha$-Gauss curvature flow.

Gauss curvature flows have also been investigated in curved ambient spaces and in the complete noncompact setting.  McCoy \cite{McC18} studied the Gauss curvature flow for convex surfaces in the three-sphere.  The $\alpha$-Gauss curvature flow for surfaces in three-dimensional space forms was studied by Hu, Li, Wei, and Zhou \cite{HLWZ20}.  More recently, Chen--Huang \cite{CH2024} considered the $\alpha$-Gauss curvature flow \eqref{eq:closed-alpha-intro} in higher-dimensional space forms.  By projecting the evolving hypersurface onto a Euclidean tangent space, they proved finite-time contraction for every $\alpha>0$ and spherical asymptotics when $\alpha>1/(n+2)$.  In the complete noncompact Euclidean setting, K. Choi, Daskalopoulos, Kim, and Lee \cite{CDKL19} proved global existence for complete, noncompact, locally uniformly convex hypersurfaces for every $\alpha>0$.  The asymptotic behavior for $\alpha>{1}/{2}$ was subsequently studied by B. Choi, K. Choi, and Daskalopoulos \cite{CCD22}.  They also classified ancient solutions to the Gauss curvature flow that are contained in a cylinder with bounded cross-section \cite{CCD24}.

By contrast, although free-boundary mean curvature type flows
have been studied extensively, fully nonlinear free-boundary
flows driven by powers of the Gauss curvature remain much less
understood. Recently, Mei, Wang and Weng \cite{MWW25} studied the
capillary Gauss curvature flow in the Euclidean half-space for the
power $\alpha=1$. They proved that the normalized flow converges to a
soliton, thereby extending the entropy method of Guan--Ni
\cite{GN17} to the capillary setting.

The free-boundary problem in the unit ball presents a different
difficulty. The support hypersurface is curved, and flattening it by a
conformal transformation changes the curvature speed. More precisely,
the transformed speed contains a curvature-dependent perturbation
which is not available in the closed Euclidean problem or in a flow
whose support hypersurface is already flat. One of the main purposes
of this paper is to control this perturbation without assuming
curvature estimates in advance. Our result proves the hemispherical asymptotics for the normalized free-boundary flow in the same natural
superaffine range  $\alpha>\frac{1}{n+2}$ as in the closed
Euclidean setting.

\subsection{Main result} 

Let $M^n$ be a compact, smooth, orientable manifold
with smooth nonempty boundary, and let $X_0:M\rightarrow \ol\B^{n+1}$  
be a smooth proper embedding in the unit ball in the sense that $ X_0(\operatorname{int}M)\subset\B^{n+1}$ and $X_0(\partial M)\subset\Sph^n$. Set
\begin{equation*}
    M_0=X_0(M),
    \qquad
    \partial M_0=X_0(\partial M)\subset\Sph^n.
\end{equation*}
Assume that $M_0$ is strictly convex and meets the support sphere
$\Sph^n=\partial\B^{n+1}$ orthogonally. We consider a smooth one-parameter family of proper embeddings
\begin{equation*}
    X(\cdot,t):M\longrightarrow\ol\B^{n+1},
    \qquad
    M_t=X(M,t),
\end{equation*}
with $X(\cdot,0)=X_0$. For each $t$ we assume that $M_t$, together with a
relatively open region of the support sphere, bounds a convex body
$\Omega_t\subset\ol\B^{n+1}$. We choose $\nu$ to be the outward unit
normal of $\Omega_t$ along $M_t$. Along the free boundary, the outward
unit normal of the support sphere is $ \mu=X$.  
The orthogonality condition implies that $\mu$ is tangent to $M_t$ and
is the outward unit conormal of $\partial M_t$ in $M_t$.  For a fixed power $\alpha>0$, we study the initial-boundary value
problem
\begin{equation}\label{eq:flow}
\begin{cases}
    \partial_tX=-K^\alpha\nu,
        & \text{in }M\times[0,T),\\
    X(\cdot,0)=X_0,
        & \\
    X(\partial M,t)\subset\Sph^n,
        & \\
    \langle\nu,\mu\rangle=0,
        & \text{on }\partial M\times[0,T),
\end{cases}
\end{equation}
where $K$ denotes the Gauss curvature of $M_t$.

Short-time existence and uniqueness for \eqref{eq:flow} follow
from the standard reduction to a scalar parabolic
initial-boundary value problem. More precisely, we can use the
generalized Gaussian coordinates introduced by Stahl
\cite[Section~2]{StahlRegularity}, which transform the geometric
free-boundary condition into a transversal homogeneous Neumann
condition. Although Stahl considered the mean curvature flow,
the coordinate construction and the boundary operator are
independent of the particular curvature speed. For the speed $G=K^\alpha$, the linearization with respect to the second fundamental form is $\dot G^{ij}=\alpha K^\alpha b^{ij}$.  Hence strict convexity of the initial hypersurface makes the
linearized equation uniformly parabolic. The
implicit function theorem argument in
\cite[Theorem~2.1]{StahlRegularity}, with the mean curvature
speed replaced by $K^\alpha$, therefore gives a unique solution
on a short time interval. The solution is smooth for positive
time and remains strictly convex for sufficiently short time.

Our main theorem is the following.

\begin{theorem}\label{thmA}
Let $n\geq2$ and $\alpha>0$. Let $M_0\subset\ol\B^{n+1}$ be a
smooth, strictly convex hypersurface with free boundary on
$\Sph^n$. Then the maximal smooth solution $M_t$ of
\eqref{eq:flow} exists on a finite time interval $[0,T)$.
The hypersurfaces $M_t$ remain strictly convex and contract in
Hausdorff distance to a single point 
$p\in\Sph^n$ as $t\to T$.

If in addition $\alpha>\frac{1}{n+2}$,  then after rotating $p$ to $-e_{n+1}$, applying the Cayley-type
conformal map to the Euclidean half-space, and rescaling so that the
enclosed half-space volume is equal to that of the unit half-ball,
the normalized hypersurfaces converge smoothly to the unit hemisphere
centered at the origin. Equivalently, the corresponding normalized
support functions converge smoothly to the constant function $1$ on
$\Sph^n_+$.
\end{theorem}

The theorem has two complementary parts. The finite-time
contraction holds for every $\alpha>0$, while, under the
additional assumption $\alpha>\frac{1}{n+2}$, the normalized
flow exhibits smooth hemispherical asymptotics in the same
natural superaffine range as in the closed Euclidean setting.

\subsection{Discussion of the proof}
The proof has two main parts. In the first part, we establish finite-time
contraction of the flow \eqref{eq:flow} to a single point for every $\alpha>0$.
The spherical free-boundary condition gives the compatibility identity
\begin{equation*}
    \nabla_\mu K=\frac{1}{\alpha}K.
\end{equation*}
Applied to the harmonic curvature $\tr(h^{-1})$, this identity prevents a boundary maximum
and yields a time-independent positive lower bound for every principal
curvature. To obtain the complementary upper curvature estimate, we introduce
the boundary-adapted Tso quotient
\begin{equation*}
    f
    =
    e^{2\rho_{B}}
    \frac{K^\alpha}{u_p-\sigma},
    \qquad
    \rho_{B}=\frac{1-|X|^2}{2},
    \qquad
    u_p=\langle X-p,\nu\rangle.
\end{equation*}
The factor $e^{2\rho_{B}}$ makes the outward conormal derivative of
$f$ strictly negative on the free boundary. Hence a new maximum of
$f$ must occur in the interior. If the inradius remained positive up
to the maximal time, these estimates would give uniform two-sided
curvature bounds and the solution could be extended. It follows that
the inradius tends to zero. A radius comparison of Chou--Wang \cite[Lem 2.2]{CW00},
combined with the uniform lower curvature bound, then shows that the
circumradius also tends to zero. This rules out collapse to a
nontrivial segment and proves contraction to a single point on the
support sphere.

In the second part, we study the asymptotic shape near the extinction
point. After rotating the extinction point to $-e_{n+1}$, we use a
Cayley-type conformal map which sends the unit ball to the Euclidean
half-space, the extinction point to the origin, and the support sphere
to the flat wall $ \Pi$. 
The free-boundary condition becomes the exact Neumann condition
\begin{equation*}
    Z(\partial M,t)\subset\Pi,
    \qquad
    \langle N,e_{n+1}\rangle=0.
\end{equation*}
Thus  even reflection across $\Pi$ becomes available. See section \ref{sec:normalization} for the details.

\begin{figure}[h]
  \centering
  \includegraphics[width=0.9\textwidth]{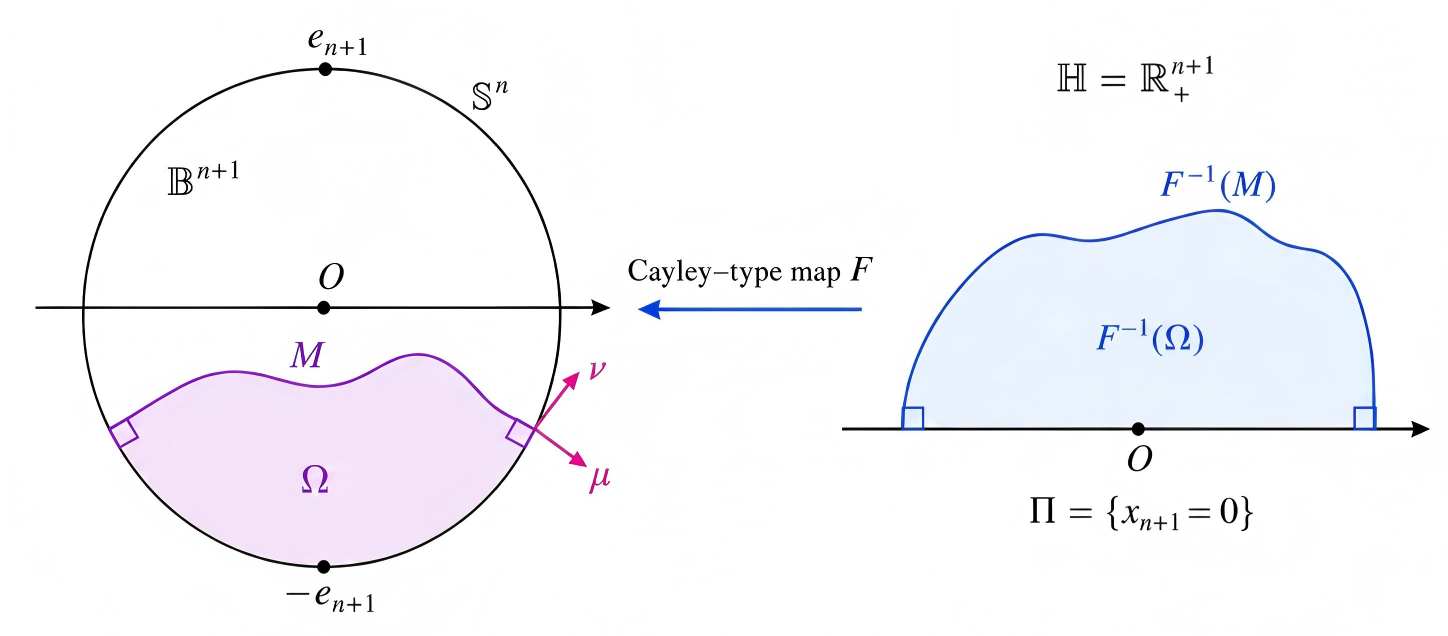}
  \caption{The Cayley-type conformal map}
\end{figure}

The price of this simplification is that the transformed speed is no
longer exactly $K^\alpha$. After
normalizing the enclosed half-space volume, the support function
satisfies
\begin{equation*}
    \partial_s h
    =
    h-
    \frac{\Theta K^\alpha}
    {\displaystyle
     \frac{1}{|\Sph^n_+|}
     \int_{\Sph^n_+}\Theta K^{\alpha-1}\,d\sigma},
    \qquad
    \overline{\nabla}_\eta h=0\quad \text{on}~\partial \mathbb{S}^n_+
\end{equation*}
where
\begin{equation*}
    \Theta
    =
    \Lambda^{n\alpha+1}
    \det(\Id+\rho qB)^\alpha.
\end{equation*}
Although $\Theta$ tends formally to one, it contains the curvature
tensor $B$. Consequently, a pointwise estimate for $\Theta$ would
require the very $C^2$ bounds that the entropy argument is intended
to produce. This circularity is the main analytic difficulty of the
paper, and differs significantly from both the closed Euclidean setting \cite{AGN2016,BCD2017,GN17} and the space-form treatment of Chen--Huang \cite{CH2024}.

We overcome it by proving an almost-monotonicity formula for a
half-space entropy. Instead of estimating
$\det(\Id+\rho qB)$ pointwise, we control its contribution after
integration by parts. The resulting error is bounded by the
unnormalized diameter, which decays integrably in the normalized time.
The argument takes different forms in the two ranges of $\alpha$. If
$\alpha\geq1$, the gradient term produced by integration by parts has
a favorable sign. If
\begin{equation*}
    \frac{1}{n+2}<\alpha<1,
\end{equation*}
the same term has the opposite sign and may become singular. In this
range, we first establish a free-boundary Minkowski--Reilly inequality
and use it to prove monotonicity of an Andrews-type quantity \cite[\S 4]{And00}. This
gives a preliminary entropy bound. We then transport the entropy point
backward from a terminal time, which provides uniform positive upper
and lower bounds for the translated support function on a unit time
strip. These estimates make the integration by parts argument
possible.

The almost-monotonicity yields the existence of the entropy limit and
shows that the tangential entropy point converges to the origin.
Together with the static entropy geometry, this gives uniform
$C^0$ and $C^1$ bounds for the normalized support function. Shifted
Tso estimates provide two-sided bounds for the transformed speed. A
spectral perturbation argument, based on
\cite[Lemma~5]{BCD2017}, then controls the largest principal radius and
gives uniform $C^2$ estimates. Higher-order estimates follow from
parabolic regularity. Every subsequential limit satisfies the exact
half-space soliton equation. After even reflection across $\Pi$, the
limit becomes a smooth closed uniformly convex shrinker. The
classification theorem of Brendle--Choi--Daskalopoulos \cite{BCD2017}
therefore implies that the limit is the unit sphere, and the original
free-boundary limit is the unit hemisphere. This completes the proof of Theorem \ref{thmA}.

\subsection{Organization of the paper}  Section~\ref{sec:prelim} derives the boundary geometry and evolution equations. Section~\ref{sec:extinction} proves finite-time contraction of the flow \eqref{eq:flow} to a single boundary point for every $\alpha>0$. Section~\ref{sec:normalization} introduces conformal half-space coordinates and volume normalization, yielding the normalized flow~\eqref{eq:support-flow-alpha} on the hemisphere. Section~\ref{sec:entropy-alpha} establishes the almost-monotonicity of the half-space entropy and derives the $C^0$ and $C^1$ bounds on $h$ via convergence of the entropy point. Section~\ref{sec:apriori-alpha} proves uniform curvature bounds and higher-order regularity for the normalized flow \eqref{eq:support-flow-alpha}. Finally, in Section~\ref{sec:convergence-alpha}, we complete the proof of smooth convergence to the unit hemisphere.

\section{Boundary geometry and evolution equations}
\label{sec:prelim}

In this paper, Latin indices range from $1$ to $n$, whereas Greek indices range from $2$ to $n$ and denote directions tangent to $\bd M_t$.  At the boundary, we take $e_1=\mu=X$ and choose $e_\alpha$ tangent to $\bd M_t$.  The second fundamental form is
\begin{equation*}
        h_{ij}=\langle D_{e_i}\nu,e_j \rangle,
\end{equation*}
with Weingarten map $h_i^j=g^{jk}h_{ik}$ and principal curvatures $\kappa_1,\ldots,\kappa_n>0$.  The Gauss curvature is $K=\det(h_i^j)$.  Set
\begin{equation*}
        G=K^\alpha,
        \qquad
        \dot G^{ij}=\frac{\bd G}{\bd h_{ij}},
        \qquad
        \ddot G^{ij,kl}=\frac{\bd^2G}{\bd h_{ij}\bd h_{kl}},
\end{equation*}
and define the linearized parabolic operator
\begin{equation}\label{eq:Lalpha}
        \cL=\partial_t-\dot G^{ij}\nabla_i\nabla_j.
\end{equation}
At a point where $h_i^j$ is diagonal,
\begin{equation}\label{eq:G-derivatives-diag}
        \dot G^{ij}=\alpha K^\alpha\kappa_i^{-1}\delta^{ij},
        \qquad
        \dot G^{ij}(h^2)_{ij}=\alpha K^\alpha H.
\end{equation}
The speed $G$ is homogeneous of degree $n\alpha$ in the principal curvatures: $\dot G^{ij}h_{ij}=n\alpha G$.  

\subsection{Boundary identities}

\begin{lemma}\label{lem:bd-identities-alpha}
At each point of $\partial M_t$, take an orthonormal frame $\{e_1=\mu,e_2,\dots,e_n\}$. We have the following identities:
\begin{enumerate}[label=\textup{(\roman*)}]
\item $\mu=X$ and $\mu\in TM_t$.
\item $h_{\mu \gamma}=0$ for every $\gamma=2,...,n$.
\item For every $\beta,\gamma=2,\dots,n$,
\begin{equation}\label{eq:bd-hab-alpha}
        \nabla_\mu h_{\beta\gamma}=-h_{\beta\gamma}+h_{\mu\mu}g_{\beta\gamma}.
\end{equation}
\item For the speed $G=K^\alpha$, we have
\begin{equation}\label{eq:bd-G-alpha}
        \nabla_\mu G=G,
        \qquad\hbox{equivalently}\qquad
        \nabla_\mu K=\frac{1}{\alpha} K.
\end{equation}
\item For every fixed $p\in\R^{n+1}$ and $u_p=\langle X-p,\nu\rangle$,
\begin{equation}\label{eq:bd-up-alpha}
        \nabla_\mu u_p=h((X-p)^T,\mu)=h_{\mu\mu}(1-\ip{p}{X}).
\end{equation}
In particular, if $p\in\ol\B^{n+1}$, then $\nabla_\mu u_p\ge0$.
\end{enumerate}
\end{lemma}

\begin{proof}
The first three statements do not depend on the speed and may be found in \cite[\S 2]{sta96}.

To prove \eqref{eq:bd-G-alpha}, observe that, along $\bd M_t$, the vector $\partial_tX=-G\nu$ is tangent to the support sphere.  The time derivative of the sphere normal is therefore $\partial_t\mu=\partial_tX=-G\nu$, while $\partial_t\nu=\nabla G$.  Hence
\begin{equation*}
        0=\partial_t\ip{\nu}{\mu}
        =\ip{\nabla G}{\mu}+\ip{\nu}{-G\nu}
        =\nabla_\mu G-G.
\end{equation*}
Since $G=K^\alpha$ and $\alpha>0$, this is equivalent to $\nabla_\mu K=K/\alpha$.

Finally,
\begin{align*}
        \nabla_\mu u_p=&\nabla_\mu \ip{X-p}{\nu}\\
        =&\ip{\mu}{\nu}+\ip{X-p}{D_\mu\nu}
        =h((X-p)^T,\mu).
\end{align*}
Using $h_{\mu \beta}=0$ for $\beta=2,\dots,n$ and $\mu=X$, we obtain
\begin{align*}
   \nabla_\mu u_p=& \sum_{\beta=1}^n \langle X-p,e_{\beta}\rangle  h_{\mu \beta}\\
   =&h_{\mu\mu}\langle X-p,X\rangle =h_{\mu\mu}(1-\langle p,X\rangle).
\end{align*}
\end{proof}

\begin{lemma}\label{lem:dm-hmm-alpha}
At a point of $\partial M_t$, choose an orthonormal frame with $e_1=\mu$ and $h_{\beta\gamma}=\kappa_{\beta}\delta_{\beta\gamma}$. Then
\begin{equation*}
        \nabla_\mu h_{\mu\mu}
        =h_{\mu\mu}\left(\frac{1}{\alpha}-
        \sum_{\gamma=2}^n\frac{h_{\mu\mu}-\kappa_{\gamma}}{\kappa_{\gamma}}\right).
\end{equation*}
\end{lemma}

\begin{proof}
At the boundary,
\begin{equation*}
        K=h_{\mu\mu}\prod_{\gamma=2}^nh_{\gamma\gamma}.
\end{equation*}
Using $\nabla_\mu K=K/\alpha$ and \eqref{eq:bd-hab-alpha},
\begin{equation*}
        \frac1\alpha
        =\frac{\nabla_\mu h_{\mu\mu}}{h_{\mu\mu}}
        +\sum_{\gamma=2}^n\frac{\nabla_\mu h_{\gamma\gamma}}{h_{\gamma\gamma}}
        =\frac{\nabla_\mu h_{\mu\mu}}{h_{\mu\mu}}
        +\sum_{\gamma=2}^n\frac{h_{\mu\mu}-\kappa_\gamma}{\kappa_\gamma}.
\end{equation*}
This proves the formula.
\end{proof}

\subsection{Interior evolution equations}

\begin{lemma}\label{lem:evol-alpha}
Along \eqref{eq:flow}, with $\cL$ as in \eqref{eq:Lalpha},
\begin{align}
        \cL G&=G\dot G^{ij}(h^2)_{ij}=\alpha K^{2\alpha}H,\label{eq:evol-G}\\
        \cL h_i^j&=\ddot G^{pq,rs}\nabla_i h_{rs}\nabla^j h_{pq}
        +\dot G^{pq}(h^2)_{pq}h_i^j+(1-n\alpha)G(h^2)_i^j, \label{eq:hij}\\
        \cL u_p&=\dot G^{ij}(h^2)_{ij}u_p-(n\alpha+1)G
        =\alpha K^\alpha H u_p-(n\alpha+1)K^\alpha,\label{eq:evol-u-alpha}\\
\partial_td\mu_t&= -HGd\mu_t,\label{eq:dmu}\\
        \partial_t|X-p|^2&=-2G u_p.\label{eq:evol-dist-alpha}
\end{align}
Moreover, for $\rho_B=(1-|X|^2)/2$,
\begin{align}
\label{eq:evol-rho-alpha}
        \cL\rho_B&=\dot G^{ij}g_{ij}+(1-n\alpha)G\ip{X}{\nu},\\
\label{eq:bd-rho-alpha}
        \nabla_\mu\rho_B&=-1\qquad\hbox{on }\bd M_t.
\end{align}
\end{lemma}

\begin{proof}
For a curvature flow $\partial_tX=-G\nu$ of hypersurfaces, the evolution equations for the speed $G$, the Weingarten tensor $h_i^j$ and the area element $d\mu_t$ are standard, see \cite{And94,CDKL19}.

For the equation \eqref{eq:evol-u-alpha} of the support function $u_p$, the time derivative is 
\begin{equation*}
    \partial_tu_p=\langle \partial_tX,\nu\rangle +\langle X-p,\partial_t\nu\rangle =-G+\langle X-p,\nabla G\rangle. 
\end{equation*}
The spatial derivatives are
\begin{align*}
       \nabla_iu_p=&\langle X-p,h_i^ke_k\rangle \\
       \nabla_j\nabla_iu_p=&h_{ij}-h_i^kh_{kj}u_p+\langle X-p,\nabla h_{ij}\rangle . 
\end{align*}
Combining these identities gives
\begin{align*}
    \cL u_p=&-G-\dot{G}^{ij}h_{ij}+\dot{G}^{ij}(h^2)_{ij}u_p\\
    =&\dot{G}^{ij}(h^2)_{ij}u_p-(n\alpha+1)G\\
    =&\alpha K^\alpha H u_p-(n\alpha+1)K^\alpha.
\end{align*}
Equation \eqref{eq:evol-dist-alpha} is immediate from
\begin{equation*}
        \partial_t|X-p|^2=2\ip{X-p}{-G\nu}=-2Gu_p.
\end{equation*}

For $\rho_B$, since $D\rho_B=-X$ and $D^2\rho_B=-\ol g$,
\begin{equation*}
        \partial_t\rho_B=\ip{-X}{-G\nu}=G\ip{X}{\nu},
\end{equation*}
and
\begin{equation*}
        \nabla_i\nabla_j\rho_B=-g_{ij}+h_{ij}\ip{X}{\nu}.
\end{equation*}
Therefore
\begin{equation*}
        \cL\rho_B=G\ip{X}{\nu}-\dot G^{ij}(-g_{ij}+h_{ij}\ip{X}{\nu})
        =\dot G^{ij}g_{ij}+(1-n\alpha)G\ip{X}{\nu}.
\end{equation*}
The boundary derivative follows from $\mu=X$.
\end{proof}

\section{Finite-time contraction to a boundary point}
\label{sec:extinction}

In this section we prove that the flow \eqref{eq:flow} has finite maximal time and contracts to a single point on $\mathbb{S}^n$.

\subsection{A lower bound for the speed}
\begin{lemma}\label{lem:K-lower-alpha}
Let $K_0$ be the Gauss curvature of the initial strictly convex free-boundary hypersurface $M_0$, and suppose that $M_t=X(M,t)$ solves \eqref{eq:flow} on $[0,T)$.  Then
\begin{equation}\label{eq:K-lower-alpha}
        K(x,t)\ge
        \left(\big(\min_MK_0\big)^{-\frac{n\alpha+1}{n}}-(n\alpha+1)t\right)^{-\frac{n}{n\alpha+1}}
\end{equation}
for all $(x,t)\in M\times [0,T)$.  In particular, the maximal existence time of the flow \eqref{eq:flow} satisfies
\begin{equation*}
        T\le \frac1{n\alpha+1}\big(\min_MK_0\big)^{-\frac{n\alpha+1}{n}}.
\end{equation*}
\end{lemma}

\begin{proof}
We apply the maximum principle to the evolution equation \eqref{eq:evol-G} for $G=K^\alpha$.  Since $H\geq nK^{1/n}$, we have
\begin{equation*}
        \cL G=\alpha K^{2\alpha}H\ge n\alpha K^{2\alpha+1/n}
        =n\alpha G^{2+1/(n\alpha)}.
\end{equation*}
Set $G_0=\min_MK_0^\alpha$ and
\begin{equation*}
        y(t)=\left(G_0^{-\frac{n\alpha+1}{n\alpha}}-(n\alpha+1)t\right)^{-\frac{n\alpha}{n\alpha+1}}.
\end{equation*}
Then
\begin{equation*}
  \dfrac{\mathrm{d}y}{\mathrm{d}t}=n\alpha y^{2+1/(n\alpha)}
\end{equation*}
and 
\begin{equation*}
        \mathcal{L}(G-y)\ge n\alpha\big(G^{2+1/(n\alpha)}-y^{2+1/(n\alpha)}\big).
\end{equation*}
Initially, $G-y\ge0$. A first interior zero would contradict the parabolic maximum principle. At a first boundary zero, \eqref{eq:bd-G-alpha} gives
\begin{equation*}
        \nabla_\mu(G-y)=\nabla_\mu G=G>0,
\end{equation*}
whereas the outward normal derivative at a boundary minimum must be nonpositive. Thus $G\ge y$, which is equivalent to \eqref{eq:K-lower-alpha}.  The upper bound for $T$ follows from the blow-up time of $y$.
\end{proof}

\subsection{A lower bound for the principal curvatures}
\begin{lemma}\label{lem:kappa-lower-alpha}
Let $M_t$ be a solution of \eqref{eq:flow} on $[0,T)$.  There exists $c_1=c_1(n,\alpha,M_0)>0$ such that
\begin{equation}\label{eq:kappa-lower-alpha}
        \kappa_i(x,t)\ge c_1,
        \qquad i=1,\ldots,n,
\end{equation}
for all $(x,t)\in M\times[0,T)$.
\end{lemma}

\begin{proof}
We use the reciprocal mean curvature quantity
\begin{equation*}
        \Phi=\tr(h^{-1})=b_i{}^i=\sum_{i=1}^n\frac{1}{\kappa_i},
        \qquad b_i{}^j=(h^{-1})_i{}^j.
\end{equation*}
A uniform upper bound for $\Phi$ is equivalent to a uniform positive lower bound for all principal curvatures.

We only need to verify that the free boundary cannot create a new maximum of $\Phi$.  At a boundary point choose an orthonormal frame with $e_1=\mu$ and $h_{\beta\gamma}=\kappa_\beta\delta_{\beta\gamma}$. Since $h_{\mu\beta}=0$, the eigenvalue in the conormal direction is $h_{\mu\mu}$. Differentiating $\Phi$ in the outward conormal direction gives
\begin{equation*}
        \nabla_\mu\Phi
        =-\frac{\nabla_\mu h_{\mu\mu}}{h_{\mu\mu}^2}
          -\sum_{\gamma=2}^n\frac{\nabla_\mu h_{\gamma\gamma}}{\kappa_\gamma^2}.
\end{equation*}
By Lemma \ref{lem:dm-hmm-alpha},
\begin{equation*}
        \frac{\nabla_\mu h_{\mu\mu}}{h_{\mu\mu}}
        =\frac1\alpha-
        \sum_{\gamma=2}^n\frac{h_{\mu\mu}-\kappa_\gamma}{\kappa_\gamma},
\end{equation*}
and by \eqref{eq:bd-hab-alpha}, $\nabla_\mu h_{\gamma\gamma}=h_{\mu\mu}-\kappa_\gamma$.  Therefore
\begin{equation}\label{eq:Dphi}
\begin{aligned}
        \nabla_\mu\Phi
        &=-\frac1{\alpha h_{\mu\mu}}
          +\sum_{\gamma=2}^n\frac{h_{\mu\mu}-\kappa_\gamma}{h_{\mu\mu}\kappa_\gamma}
          -\sum_{\gamma=2}^n\frac{h_{\mu\mu}-\kappa_\gamma}{\kappa_\gamma^2}  \\
        &=-\frac1{\alpha h_{\mu\mu}}
          -\sum_{\gamma=2}^n
          \frac{(h_{\mu\mu}-\kappa_\gamma)^2}{h_{\mu\mu}\kappa_\gamma^2}
        \le0.
\end{aligned}
\end{equation}
Thus, any first maximum of $\Phi$ that exceeds its initial maximum must occur in the interior.  

We next compute the interior inequality in indices.  Fix an interior point and choose a local orthonormal frame which is normal at this point and diagonalizes the Weingarten map.  Thus
\begin{equation*}
        h_i{}^j=\kappa_i\delta_i{}^j,
        \qquad
        b_i{}^j=\kappa_i^{-1}\delta_i{}^j.
\end{equation*}
The identity $b_i{}^p h_p{}^j=\delta_i{}^j$ gives, after differentiating once in space or time,
\begin{equation*}
        \partial_t b_i{}^j=-b_i{}^p(\partial_t h_p{}^q)b_q{}^j,
        \qquad
        \nabla_k b_i{}^j=-b_i{}^p(\nabla_k h_p{}^q)b_q{}^j.
\end{equation*}
Differentiating the spatial identity once more gives
\begin{align*}
        \nabla_k\nabla_l b_i{}^j
        ={}&-b_i{}^p(\nabla_k\nabla_l h_p{}^q)b_q{}^j  \\
        &+b_i{}^a(\nabla_kh_a{}^p)b_p{}^r(\nabla_lh_r{}^q)b_q{}^j  \\
        &+b_i{}^p(\nabla_lh_p{}^q)b_q{}^a(\nabla_kh_a{}^r)b_r{}^j .
\end{align*}
Applying $\cL=\partial_t-\dot G^{kl}\nabla_k\nabla_l$ and using the symmetry of $\dot G^{kl}$, we obtain
\begin{align*}
        \cL b_i{}^j
        ={}&-b_i{}^p b_q{}^j\,\cL h_p{}^q  -2\dot G^{kl}b_i{}^p b_q{}^r b_s{}^j
        (\nabla_k h_p{}^q)(\nabla_l h_r{}^s).
\end{align*}
Taking the trace yields
\begin{equation}\label{eq:L-b-trace-alpha}
        \cL\Phi
        =-b_i{}^p b_q{}^i\,\cL h_p{}^q
        -2\dot G^{kl}b_i{}^p b_q{}^r b_s{}^i
        (\nabla_k h_p{}^q)(\nabla_l h_r{}^s).
\end{equation}

We now substitute the evolution equation \eqref{eq:hij}.  The zero-order part is
\begin{equation}\label{eq:Phi-zero-alpha}
\begin{aligned}
 &-b_i{}^p b_q{}^i
 \left(\dot G^{rs}(h^2)_{rs}h_p{}^q+(1-n\alpha)G(h^2)_p{}^q\right)  \\
 &\qquad
 =-\dot G^{rs}(h^2)_{rs}\,\tr(h^{-1})
   -(1-n\alpha)G\,\tr(b^2h^2)  \\
 &\qquad
 =-\alpha K^\alpha H\Phi+n(n\alpha-1)K^\alpha.
\end{aligned}
\end{equation}
Here we used $\tr h^{-1}=\Phi$, $\tr (b^2h^2)=n$, and
\begin{equation*}
        \dot G^{rs}(h^2)_{rs}=\alpha K^\alpha H.
\end{equation*}

It remains to check the gradient terms $\mathcal G$.  For $G=K^\alpha$, at the diagonal point,
\begin{equation*}
        \ddot G^{ij,kl}\eta_{ij}\eta_{kl}
        =\alpha K^\alpha
        \left[\alpha\left(\sum_i\frac{\eta_{ii}}{\kappa_i}\right)^2
        -\sum_{i,j}\frac{\eta_{ij}^2}{\kappa_i\kappa_j}\right]
\end{equation*}
for every symmetric two-tensor $\eta$.  Hence the contribution from the $\ddot G$ term in \eqref{eq:hij} is
\begin{align*}
 &-b_i{}^p b_q{}^i\ddot G^{ab,cd}
        (\nabla_p h_{ab})(\nabla^q h_{cd})  \\
 &\quad
 =-\sum_p\frac1{\kappa_p^2}\ddot G^{ab,cd}
        (\nabla_p h_{ab})(\nabla_p h_{cd})  \\
 &\quad
 =-\alpha^2K^\alpha\sum_p\frac1{\kappa_p^2}
        \left(\sum_i\frac{\nabla_p h_{ii}}{\kappa_i}\right)^2
   +\alpha K^\alpha\sum_{p,i,j}
        \frac{(\nabla_p h_{ij})^2}{\kappa_p^2\kappa_i\kappa_j}.
\end{align*}
The second gradient term in \eqref{eq:L-b-trace-alpha} becomes
\begin{align*}
 &-2\dot G^{kl}b_i{}^p b_q{}^r b_s{}^i
        (\nabla_k h_p{}^q)(\nabla_l h_r{}^s)  \\
 &\quad
 =-2\alpha K^\alpha\sum_{p,i,j}
        \frac{(\nabla_p h_{ij})^2}{\kappa_p\kappa_i^2\kappa_j}.
\end{align*}
Combining the two displayed formulas gives
\begin{align*}
\mathcal G
={}&-\alpha^2K^\alpha\sum_p\frac1{\kappa_p^2}
        \left(\sum_i\frac{\nabla_p h_{ii}}{\kappa_i}\right)^2 \\
&+\alpha K^\alpha\sum_{p,i,j}
        \left(\frac1{\kappa_p^2\kappa_i\kappa_j}
        -\frac2{\kappa_p\kappa_i^2\kappa_j}\right)
        (\nabla_p h_{ij})^2 .
\end{align*}
The second line is non-positive after symmetrization.  Indeed, by Codazzi equations,  $\nabla_p h_{ij}=\nabla_i h_{pj}$.  Therefore
\begin{align*}
&\sum_{p,i,j}
        \left(\frac1{\kappa_p^2\kappa_i\kappa_j}
        -\frac2{\kappa_p\kappa_i^2\kappa_j}\right)(\nabla_p h_{ij})^2  \\
&\quad
=\frac12\sum_{p,i,j}
        \left[
        \left(\frac1{\kappa_p^2\kappa_i\kappa_j}
        -\frac2{\kappa_p\kappa_i^2\kappa_j}\right)
        +
        \left(\frac1{\kappa_i^2\kappa_p\kappa_j}
        -\frac2{\kappa_i\kappa_p^2\kappa_j}\right)
        \right](\nabla_p h_{ij})^2  \\
&\quad
=-\frac12\sum_{p,i,j}
        \left(\frac1{\kappa_p^2\kappa_i\kappa_j}
        +\frac1{\kappa_p\kappa_i^2\kappa_j}\right)(\nabla_p h_{ij})^2
\le0.
\end{align*}
Thus $\mathcal G\le0$.  Combining this with \eqref{eq:Phi-zero-alpha} gives the interior inequality
\begin{equation}\label{eq:L-Phi-alpha-detailed}
        \cL\Phi
        \le -\alpha K^\alpha H\Phi+n(n\alpha-1)_+K^\alpha.
\end{equation}

If $n\alpha\le1$, the right-hand side is non-positive at every positive interior maximum. Together with the boundary sign \eqref{eq:Dphi}, the maximum principle gives
\begin{equation*}
        \Phi(\cdot,t)\le \max_M\Phi(\cdot,0).
\end{equation*}
If $n\alpha>1$ and a new interior maximum occurs, then $0\le\cL\Phi$ at that point.  From \eqref{eq:L-Phi-alpha-detailed} we obtain
\begin{equation*}
        H\Phi\le \frac{n(n\alpha-1)}{\alpha}.
\end{equation*}
Since $H\ge nK^{1/n}$ and Lemma \ref{lem:K-lower-alpha} implies $K\ge \min_MK_0$ on $[0,T)$, the value of $\Phi$ at such a maximum is bounded by
\begin{equation*}
        \Phi\le \frac{n\alpha-1}{\alpha(\min_MK_0)^{1/n}}.
\end{equation*}
Combining this with the initial maximum gives a constant $C_1=C_1(n,\alpha,M_0)$ such that $\Phi\le C_1$ everywhere.  Hence $\kappa_i^{-1}\le \Phi\le C_1$ and \eqref{eq:kappa-lower-alpha} follows with $c_1=C_1^{-1}$.

\end{proof}

\subsection{A boundary-adapted Tso estimate}
\label{subsec:Tso-alpha}
We prove a Tso-type upper bound for the Gauss curvature under a positive lower bound for the inradius.  The following estimate for the boundary defining function $\rho_B$ excludes a boundary maximum of the modified Tso quotient.

\begin{lemma}\label{lem:rho-est-alpha}
There exists
$C=C(n,\alpha,M_0)$ such that
\begin{equation*}
        |\mathcal{L}\rho_B|\le CG,
        \qquad
        \dot G^{ij}\nabla_i\rho_B\nabla_j\rho_B\le CG.
\end{equation*}
\end{lemma}

\begin{proof}
By \eqref{eq:G-derivatives-diag} and the lower bound for the principal curvatures in \eqref{eq:kappa-lower-alpha},
\begin{equation*}
        \dot G^{ij}g_{ij}
        =\alpha G\sum_i\kappa_i^{-1}\le CG.
\end{equation*}
Together with $|\langle X,\nu\rangle|\le1$, this proves the first estimate from \eqref{eq:evol-rho-alpha}.

Since
$|\nabla\rho_B|\le|D\rho_B|=|X|\le1$, combining this inequality with \eqref{eq:kappa-lower-alpha} yields
\begin{equation*}
    \dot{G}^{ij}\nabla_i \rho_B \nabla_j \rho_B = \alpha G \sum\limits_{i=1}^n \dfrac{(\nabla_i \rho_B)^2}{\kappa_i}\leq CG.
\end{equation*}
This is the second estimate.
\end{proof}

Let $\Omega_t$ be the convex body bounded by $M_t$ and $\mathbb{S}^n$. Let $p\in\overline{\mathbb{B}}^{n+1}$, and choose $\sigma>0$ and $T_{\sigma}>0$ such that $B_{2\sigma}(p)\subset\Omega_t$ throughout $M\times [0,T_{\sigma}]$. Then $u_p=\langle X-p,\nu \rangle\geq 2\sigma$ on $M\times [0,T_{\sigma}]$.  Consider the Tso-type \cite{Tso1985} auxiliary function $f: M\times [0,T_{\sigma}]\to\mathbb{R}$ defined by
\begin{equation*}
f(x,t) = \psi \cdot \dfrac{G}{u_p-\sigma}(x,t),
\end{equation*}
where $\psi=\mathrm{e}^{2\rho_B}$.  We compute the evolution of $f$ along \eqref{eq:flow}.

\begin{lemma}\label{lem:Tso-quotient-alpha}
Along \eqref{eq:flow}, the function $f$ satisfies
\begin{align}
\mathcal{L} f=&2\frac{\psi}{u_p-\sigma}\dot G^{ij}
        \nabla_i\!\left(\frac{u_p-\sigma}{\psi}\right)\nabla_jf
        +4f\dot G^{ij}\nabla_i\rho_B\nabla_j\rho_B\nonumber\\
        &-\frac{4f}{u_p-\sigma}\dot G^{ij}\nabla_i\rho_B\nabla_j u_p + 2f\mathcal{L}\rho_B\nonumber\\
        & + (n\alpha+1)\dfrac{fG}{u_p-\sigma} - \sigma\alpha\dfrac{fGH}{u_p-\sigma}.
        \label{eq:Lf-drift-alpha}
\end{align}
On the free boundary $\partial M\times[0,T_{\sigma}]$, the normal derivative satisfies
\begin{equation}
\label{eq:Tso-boundary-derivative}
\nabla_{\mu}f\leq -f.
\end{equation}
\end{lemma}
\begin{proof}
    On the free boundary, \eqref{eq:bd-G-alpha}, \eqref{eq:bd-up-alpha}, and
\eqref{eq:bd-rho-alpha} give
\begin{equation*}
        \frac{\nabla_\mu f}{f}
        =\frac{\nabla_\mu G}{G}
          -\frac{\nabla_\mu u_p}{u_p-\sigma}
          +2\nabla_\mu\rho_B=-1-\frac{\nabla_\mu u_p}{u_p-\sigma}.
\end{equation*}
Since $p\in\overline{\B}^{n+1}$, \eqref{eq:bd-up-alpha} gives
$\nabla_\mu u_p\ge0$, and \eqref{eq:Tso-boundary-derivative} follows.

It remains to compute the interior quotient identity.  Direct differentiation yields
\begin{align*}
    \mathcal{L} f=&2\frac{\psi}{u_p-\sigma}\dot G^{ij}
        \nabla_i\!\left(\frac{u_p-\sigma}{\psi}\right)\nabla_j f
        +\frac{f}{G}\mathcal{L} G-\frac{f}{u_p-\sigma} \mathcal{L} u_p
        \nonumber\\
        &  +2f \mathcal{L}\rho_B +4f\dot G^{ij}\nabla_i\rho_B\nabla_j\rho_B
        -\frac{4f}{u_p-\sigma}\dot G^{ij}\nabla_i\rho_B\nabla_j u_p.
\end{align*}
Substituting the evolution equations \eqref{eq:evol-G} and \eqref{eq:evol-u-alpha}, we obtain
\begin{equation*}
        \frac{f}{G}\mathcal{L} G-\frac{f}{u_p-\sigma}\mathcal{L} u_p
        =-\alpha\sigma\frac{f}{u_p-\sigma}G H
          +(n\alpha+1)\frac{f}{u_p-\sigma}G.
\end{equation*}
Combining the above identities gives \eqref{eq:Lf-drift-alpha}.
\end{proof}

We now derive upper bounds for the Gauss and principal curvatures when the convex body $\Omega_t$ bounded by $M_t$ and $\mathbb{S}^n$ contains a small ball. The exponential factor $\psi=\mathrm{e}^{2\rho_B}$ excludes a boundary maximum and preserves the sharp scaling needed in Lemma \ref{prop:F-upper-alpha} below.

\begin{lemma}[Boundary Tso estimate]\label{prop:Tso-alpha}
Let $M_t$ be a solution of \eqref{eq:flow} such that
\begin{equation*}
        u_p=\langle X-p,\nu \rangle \ge  2\sigma
\end{equation*}
for some constant $\sigma>0$ on $M\times [0,T_{\sigma}]$. Then
\begin{align}
        \max_{M\times[0,T_{\sigma}]}K^\alpha
        \leq & C\max\left\{
        \frac{1}{\sigma}\sup_MK^\alpha(\cdot,0),
        \frac{1}{\sigma^{n\alpha+1}}
        \right\}. \label{eq:G-upper-Tso-alpha}
\end{align}
Moreover, set $R_t=\sup_{M_t}|X(\cdot,t)-p|\leq 2$, then
\begin{align}
  \max\limits_{M_t}K^{\alpha}(\cdot,t)  \leq& C\max\left\{
        \frac{R_t}{\sigma}\sup_MK^\alpha(\cdot,0),
        \frac{R_t}{\sigma^{n\alpha+1}}
        \right\}. \label{equ-G-upper-time}
\end{align}
Here $C=C(n,\alpha,M_0)$.
\end{lemma}

\begin{proof}
By Lemma \ref{lem:Tso-quotient-alpha}, $\nabla_\mu f\le -f<0$ on the free boundary.  Hence a positive space-time maximum of $f$
cannot occur on $\partial M$.

Fix $\tau\in(0,T_\sigma]$ and consider $f$ on
$M\times[0,\tau]$. Suppose that its maximum is attained at an
interior point $(x_0,t_0)$ with $t_0>0$. At this point the gradient
term involving $\nabla f$ in \eqref{eq:Lf-drift-alpha} vanishes and
$\mathcal L f\geq0$.

Since $X,p\in\ol\B^{n+1}$, we have
\begin{equation*}
    \sigma\leq u_p-\sigma\leq u_p
    \leq |X-p|\leq2.
\end{equation*}
Using
\begin{equation*}
    H\geq nK^{1/n}=nG^{1/(n\alpha)},
    \qquad
    G=\frac{f(u_p-\sigma)}{\psi},
\end{equation*}
together with $1\leq\psi\leq e$, we obtain
\begin{align*}
    -\alpha\sigma\frac{fGH}{u_p-\sigma}
    &\leq
    -n\alpha\sigma
    \frac{fG^{1+\frac{1}{n\alpha}}}{u_p-\sigma} \\
    &=
    -n\alpha\sigma
    \psi^{-1-\frac{1}{n\alpha}}
    (u_p-\sigma)^{\frac{1}{n\alpha}}
    f^{2+\frac{1}{n\alpha}} \\
    &\leq
    -c\sigma^{1+\frac{1}{n\alpha}}
    f^{2+\frac{1}{n\alpha}}.
\end{align*}
By Lemma \ref{lem:rho-est-alpha} and the identity
\begin{equation*}
    \frac{G}{u_p-\sigma}=\frac{f}{\psi}\leq f,
\end{equation*}
the remaining lower-order terms satisfy
\begin{equation*}
    4f\dot G^{ij}\nabla_i\rho_B\nabla_j\rho_B
    +2f|\mathcal L\rho_B|
    +(n\alpha+1)\frac{fG}{u_p-\sigma}
    \leq Cf^2.
\end{equation*}
Moreover, in a principal frame,
\begin{equation*}
    \nabla_i\rho_B=-\langle X,e_i\rangle,
    \qquad
    \nabla_i u_p=\kappa_i\langle X-p,e_i\rangle.
\end{equation*}
It follows that
\begin{align*}
    \frac{4f}{u_p-\sigma}
    \left|
    \dot G^{ij}\nabla_i\rho_B\nabla_j u_p
    \right|
    &=
    \frac{4\alpha fG}{u_p-\sigma}
    \left|
    \sum_{i=1}^n
    \langle X,e_i\rangle
    \langle X-p,e_i\rangle
    \right| \\
    &\leq
    C\frac{fG}{u_p-\sigma}
    \leq Cf^2.
\end{align*}
Substituting these estimates into \eqref{eq:Lf-drift-alpha} at
$(x_0,t_0)$ gives
\begin{equation*}
    0\leq
    -c\sigma^{1+\frac{1}{n\alpha}}
    f^{2+\frac{1}{n\alpha}}
    +Cf^2.
\end{equation*}
Consequently,
\begin{equation*}
    f(x_0,t_0)\leq C\sigma^{-(n\alpha+1)}.
\end{equation*}

If the space-time maximum is attained at the initial time, then
\begin{equation*}
    f(\cdot,0)
    =\psi\frac{G(\cdot,0)}{u_p(\cdot,0)-\sigma}
    \leq
    \frac{e}{\sigma}\sup_M K^\alpha(\cdot,0).
\end{equation*}
We therefore have
\begin{equation*}
    \sup_{M\times[0,\tau]}f
    \leq
    C\max\left\{
    \frac{1}{\sigma}\sup_M K^\alpha(\cdot,0),
    \frac{1}{\sigma^{n\alpha+1}}
    \right\}.
\end{equation*}

Since $u_p-\sigma\leq2$ and $\psi\geq1$, it follows that
\begin{equation*}
    \sup_{M\times[0,\tau]}K^\alpha
    \leq
    C\max\left\{
    \frac{1}{\sigma}\sup_M K^\alpha(\cdot,0),
    \frac{1}{\sigma^{n\alpha+1}}
    \right\}.
\end{equation*}
Taking $\tau=T_\sigma$ proves \eqref{eq:G-upper-Tso-alpha}.

Finally, at a fixed time $t\in[0,T_\sigma]$, we use only on the
final time slice that
\begin{equation*}
    u_p-\sigma\leq u_p
    \leq |X-p|\leq R_t.
\end{equation*}
Hence
\begin{align*}
    \sup_{M_t}K^\alpha
    &=
    \sup_{M_t}\frac{f(u_p-\sigma)}{\psi} \\
    &\leq
    R_t\sup_{M\times[0,t]}f \\
    &\leq
    C\max\left\{
    \frac{R_t}{\sigma}\sup_M K^\alpha(\cdot,0),
    \frac{R_t}{\sigma^{n\alpha+1}}
    \right\}.
\end{align*}
This proves \eqref{equ-G-upper-time}.
\end{proof}

\subsection{Contraction to one point}
Let $r_-(t)$ and $r_+(t)$ be respectively the Euclidean inradius and
circumradius of $\Omega_t$, and set
\begin{equation*}
        \Omega_\infty = \bigcap_{0\le t<T}\Omega_t.
\end{equation*}
The curvature lower bound yields the following comparison between the two radii.

\begin{lemma}\label{lem:radius-comparison-alpha}
There is $C=C(n,M_0,\alpha)$ such that
\begin{equation}\label{eq:radius-comparison-alpha}
        r_+(t)^2\le Cr_-(t)
\end{equation}
for every $t<T$.
\end{lemma}

\begin{proof}
Set $ c_0=\min\{c_1,1\}$.  By Lemma \ref{lem:kappa-lower-alpha}, the curved part of $\partial\Omega_t$ has all principal curvatures
bounded below by $c_1$, while the spherical face $\partial\Omega_t\cap\mathbb S^n$ has all principal
curvatures equal to one. Since the two faces meet orthogonally along their common
boundary, the corner may be rounded while preserving a uniform
positive lower bound for the principal curvatures. We briefly
recall the standard construction. Near the edge, choose smooth
defining functions $r_1$ and $r_2$ for the curved and spherical
faces, respectively, normalized so that
\begin{equation*}
    |\nabla r_i|=1,
    \qquad
    \Omega_t=\{r_1\leq0,\ r_2\leq0\}.
\end{equation*}
Replace $\max\{r_1,r_2\}$ by a smooth convex regularized maximum
$m_\varepsilon(r_1,r_2)$. Away from an $\varepsilon$-neighborhood
of the edge, the resulting boundary agrees with the original
faces. The Hessian of $m_\varepsilon(r_1,r_2)$ is the sum of a convex
combination of the Hessians of $r_1$ and $r_2$ and a
nonnegative quadratic form involving
$\nabla r_1-\nabla r_2$. The first term preserves the curvature
lower bound in directions tangent to the edge, while the second
term controls the rounding direction. Since the two normals are
orthogonal along the edge, this control is uniform. Consequently,
after patching the construction with the original faces, we
obtain smooth closed convex bodies $\Omega_{t,\varepsilon}$ such
that $ \Omega_{t,\varepsilon}\rightarrow\Omega_t$ in Hausdorff sense and 
\begin{equation*}
    \kappa_i(\partial\Omega_{t,\varepsilon})\geq c_*
\end{equation*}
for a constant $c_*>0$ depending only on $n$ and
$c_0=\min\{c_1,1\}$, and independent of $t$ and $\varepsilon$.  Applying Chou--Wang \cite[Lemma~2.2]{CW00} to
$\Omega_{t,\varepsilon}$ gives
\begin{equation*}
    r_{+,\varepsilon}^2\leq C r_{-,\varepsilon},
\end{equation*}
where $C$ depends only on $n$ and the lower curvature bound. Since the inradius and circumradius are continuous under
Hausdorff convergence of compact convex bodies, letting
$\varepsilon\downarrow0$ proves the assertion \eqref{eq:radius-comparison-alpha}.
\end{proof}

\begin{lemma}\label{prop:finite-extinction-alpha}
For every $\alpha>0$, the maximal smooth strictly convex solution of
\eqref{eq:flow} has finite maximal time $T<\infty$. Moreover,
\begin{equation*}
        \lim_{t\nearrow T}r_-(t)
        =\lim_{t\nearrow T}r_+(t)=0,
        \qquad \Omega_\infty=\{p\}
\end{equation*}
for a point $p\in\Sph^n$.  Consequently $\Omega_t$ and $M_t$ converge to $p$ in Hausdorff distance.
\end{lemma}

\begin{proof}
The upper bound for the maximal time follows from Lemma
\ref{lem:K-lower-alpha}.  The normal velocity points into the convex body and the support sphere is fixed, so avoidance gives
\begin{equation*}
        \Omega_{t_2}\subset\Omega_{t_1},
        \qquad 0\le t_1\le t_2<T.
\end{equation*}
Thus $r_-(t)$ is nonincreasing and $\Omega_\infty$ is a nonempty compact convex set.

Suppose $\lim_{t\nearrow T}r_-(t)=r_*>0$.  Choose $0<\sigma<r_*/4$.  By compactness of the nested bodies, there is a fixed ball $B_{2\sigma}(p_0)\subset\Omega_t$ for all $t<T$.  Lemma \ref{prop:Tso-alpha} gives a uniform upper bound for $K$ on $M\times[0,T)$. Together with Lemma \ref{lem:kappa-lower-alpha}, this gives two-sided bounds for all principal curvatures. Consequently,
the scalar equation obtained from the generalized Gaussian
coordinates of Stahl \cite[Section~2]{StahlRegularity} remains
uniformly parabolic, while the free-boundary condition remains
uniformly oblique. Standard
continuation estimates for strictly convex curvature flows with
oblique boundary conditions then give uniform higher-order
estimates on terminal time strips. Equivalently, one may combine
Stahl's boundary coordinate reduction with the usual local
regularity and continuation argument for the uniformly parabolic
$K^\alpha$ equation.  It follows
that $M_t$ converges smoothly, as $t\uparrow T$, to a smooth
strictly convex hypersurface satisfying the same free-boundary
condition. Applying the short-time existence result with this
limiting hypersurface as initial data extends the solution past
$T$. This is the fully nonlinear counterpart of the continuation
argument in \cite[Section~6.4]{StahlRegularity}, and gives a
contradiction. Therefore
\begin{equation*}
        r_-(t)\longrightarrow0.
\end{equation*}
The radius comparison \eqref{eq:radius-comparison-alpha} now yields
$r_+(t)\to0$. Hence the diameter of $\Omega_t$ tends to zero, and the nested intersection consists of one point, denoted by $p$.

The free boundary is nonempty and lies on the fixed support sphere.  Choose $t_j\nearrow T$ and $y_j\in X(\partial M,t_j)$. Since
$\diam\Omega_{t_j}\to 0$, $y_j\to p$. As $y_j\in\mathbb{S}^n$ and the sphere is closed, $p\in\mathbb{S}^n$. Finally, nested compact sets with singleton intersection converge to that singleton in Hausdorff distance. The same is true for their curved boundary pieces $M_t$.
\end{proof}

\section{The normalized half-space flow}
\label{sec:normalization}
To study the asymptotics near the extinction point, we perform two coordinate changes. First, a conformal map sends the unit ball to the Euclidean half-space, transforming the spherical free boundary into a flat Neumann boundary. The resulting flow has a modified speed involving the conformal factor. Second, we rescale the hypersurfaces to normalize the half-space volume, which leads to a flow for the support function on the hemisphere.

\subsection{Conformal half-space coordinates}
\label{sec:conformal}

After a rotation we assume that the extinction point is
\begin{equation*}
        p=-E,
        \qquad E=e_{n+1}.
\end{equation*}
Let
\begin{equation*}
        \HH=\{z=(z',z_{n+1})\in\R^{n+1}:z_{n+1}\ge0\},
        \qquad
        \PiH=\bd\HH=\{z_{n+1}=0\}.
\end{equation*}
We use the Cayley-type conformal map
\begin{equation}\label{eq:cayley-alpha}
        F(z)=\left(\frac{2z'}{|z'|^2+(1+z_{n+1})^2},
        \frac{|z|^2-1}{|z'|^2+(1+z_{n+1})^2}\right),
\end{equation}
which was used in \cite{WX22} to study a locally constrained mean curvature flow with free boundary.  It maps the half-space $\HH$ conformally onto $\ol\B^{n+1}\setminus\{E\}$, sends $0$ to $-E$, and maps $\PiH$ onto $\Sph^n\setminus\{E\}$.  Its conformal factor is
\begin{equation}\label{eq:conformal-factor-alpha}
        F^*\delta_{\B}=e^{2w}\delta_{\HH},
        \qquad
        e^w=\frac2{\Lambda(z)},
        \qquad
        \Lambda(z)=|E+z|^2=1+2z_{n+1}+|z|^2.
\end{equation}
The inverse is
\begin{equation*}
        F^{-1}(x)=\left(\frac{2x'}{|x'|^2+(1-x_{n+1})^2},
        \frac{1-|x|^2}{|x'|^2+(1-x_{n+1})^2}\right).
\end{equation*}

Let $Z=F^{-1}\circ X$ and $\Sigma_t=Z(M,t)$.  For $t$ sufficiently close to $T$, the hypersurface $\Sigma_t$ is contained in a small neighborhood of the origin.

\begin{lemma}\label{lem:flat-boundary-alpha}
The free-boundary conditions in the ball become
\begin{equation}\label{eq:flat-bd-alpha}
        Z(\bd M,t)\subset\PiH,
        \qquad
        \ip{N}{E}=0\quad\hbox{on }\bd M,
\end{equation}
where $N$ is the Euclidean unit normal of $\Sigma_t$.
\end{lemma}

\begin{proof}
The first assertion follows from $F(\PiH)=\Sph^n\setminus\{E\}$.  Because $F$ is conformal, orthogonality to the support sphere is transformed into orthogonality to the flat support plane.  The normal to $\PiH$ is $E$, which proves \eqref{eq:flat-bd-alpha}.
\end{proof}

Let $A_Z$ be the Euclidean Weingarten map of $\Sigma_t\subset\mathbb{H}$, and let $\bar A$ denote the Weingarten map computed in the conformal metric $\bar g=e^{2w}\delta$.  We have the following transformation formula.
\begin{lemma}
If $\bar N=e^{-w}N$ is the $\bar{g}$-unit normal, then
\begin{equation}\label{eq:conformal-A-alpha}
        \bar A=e^{-w}\big(A_Z+N(w)\Id\big).
\end{equation}
Taking determinants gives
\begin{equation}\label{eq:conformal-K-alpha}
        \bar K=e^{-nw}\det\big(A_Z+N(w)\Id\big).
\end{equation}
The conformal factor also satisfies
\begin{equation}\label{eq:Nw-alpha}
        N(w)=-\frac{2\ip{z+E}{N}}{\Lambda(z)}.
\end{equation}
\end{lemma}
\begin{proof}
For a conformal change $\bar g=e^{2w}\delta$, the Levi-Civita connections satisfy
\begin{equation*}
        \bar D_UV=D_UV+U(w)V+V(w)U-\langle U,V\rangle D w.
\end{equation*}
Taking the $\bar g$-inner product with $\bar N=e^{-w}N$ yields
\begin{equation*}
        \bar h_{ij}=e^w\left(h_{ij}+N(w)g_{ij}\right).
\end{equation*}
Since $\bar g_{ij}=e^{2w}\delta_{ij}$, raising one index gives \eqref{eq:conformal-A-alpha}, and taking determinants gives \eqref{eq:conformal-K-alpha}. Finally, $w=\log2-\log \Lambda$ and $D \Lambda=2(z+E)$, which also yields \eqref{eq:Nw-alpha}.
\end{proof}

\begin{proposition}[Pulled-back flow]\label{prop:pulled-back-alpha}
Up to tangential reparametrization, the hypersurfaces $\Sigma_t$ satisfy
\begin{equation}\label{s4.flow-pb}
        \partial_tZ=-\Phi_\alpha(Z,N,A_Z)N,
\end{equation}
with flat free boundary \eqref{eq:flat-bd-alpha}, where
\begin{align}
        \Phi_\alpha(Z,N,A_Z)=&e^{-(n\alpha+1)w}
        \det\big(A_Z+N(w)\Id\big)^\alpha.\label{eq:Phi-alpha-explicit}
\end{align}
\end{proposition}

\begin{proof}
In the $z$ variables, the Euclidean metric of the ball pulls back to
$\bar g=e^{2w}\delta$.  The $\bar g$-unit normal is $\bar N=e^{-w}N$.
Let $V=\partial_tZ$.  Since $X=F\circ Z$ and the original flow has normal
speed $-\bar K^\alpha$ in the ball metric, the normal velocity satisfies
\begin{equation*}
        \ip{dF(V)}{\nu}_{\B}=\bar g(V,\bar N)
        =e^w\ip{V}{N}=-\bar K^\alpha.
\end{equation*}
Therefore
\begin{equation*}
        \ip{V}{N}=-e^{-w}\bar K^\alpha.
\end{equation*}
Using \eqref{eq:conformal-K-alpha}, we obtain
\begin{equation*}
\begin{aligned}
        e^{-w}\bar K^\alpha
        &=e^{-w}\left(e^{-nw}\det(A_Z+N(w)\Id)\right)^\alpha  \\
        &=e^{-(n\alpha+1)w}\det(A_Z+N(w)\Id)^\alpha.
\end{aligned}
\end{equation*}
Thus the geometric normal velocity of $\Sigma_t$ is
\begin{equation*}
        \ip{\partial_tZ}{N}=-\Phi_\alpha(Z,N,A_Z).
\end{equation*}

After a tangential reparametrization, this normal-velocity identity gives the vector equation \eqref{s4.flow-pb}.  For an arbitrary parametrization, write
\begin{equation*}
        \partial_tZ=V^{\top}-\Phi_\alpha N,
\end{equation*}
where $V^{\top}$ is tangent to $\Sigma_t$.  Let $W_t$ be the vector field on
$M$ defined by $dZ_t(W_t)=V^{\top}$.  Solving
\begin{equation*}
        \frac{d}{dt}\varphi_t=-W_t\circ\varphi_t,
        \qquad \varphi_0=\mathrm{id},
\end{equation*}
and setting $\widetilde Z=Z\circ\varphi_t$, we get
\begin{equation*}
        \partial_t\widetilde Z
        =\left(\partial_tZ+dZ_t(\partial_t\varphi_t)\right)\circ\varphi_t
        =-\Phi_\alpha(\widetilde Z,\widetilde N,\widetilde A_Z)\widetilde N.
\end{equation*}
At boundary points, $\partial_tZ$ is tangent to the fixed plane $\PiH$ and
$N$ is also tangent to $\PiH$ by \eqref{eq:flat-bd-alpha}.  Hence $V^{\top}$
is tangent to $T\Sigma_t\cap T\PiH=T(\partial\Sigma_t)$, so the above
reparametrization preserves $\partial M$.  Renaming $\widetilde Z$ as $Z$
gives \eqref{s4.flow-pb} together with the flat free-boundary condition \eqref{eq:flat-bd-alpha}.

Finally, \eqref{eq:Phi-alpha-explicit} follows from
\eqref{eq:conformal-factor-alpha} and \eqref{eq:Nw-alpha}.
\end{proof}

The next lemma shows that the conformal flow \eqref{s4.flow-pb} becomes Euclidean convex sufficiently near the extinction time.

\begin{lemma}\label{lem:large-s-gauss-alpha}
	There exist $t_0\in(0,T)$ and $c_2>0$, depending only on $n$, $\alpha$, and $M_0$, such that, for every $t\in [t_0,T)$, the conformal hypersurface
	$Z(M,t)\subset\HH$ along the flow \eqref{s4.flow-pb} is a smooth strictly convex free boundary hypersurface in the Euclidean metric with
	\begin{equation}
		A_Z \geq c_2 \mathrm{Id} \qquad\mathrm{on}\  M\times [t_0,T).
		\label{equ-AZlowbd}
	\end{equation}
	Moreover, $N(w)\leq 0$.
\end{lemma}
\begin{proof}
By Lemma \ref{prop:finite-extinction-alpha}, the convex bodies $\Omega_t$ are nested and
their intersection is the single point $p=-E$. In particular,
$p\in\Omega_t$ for every $t<T$. Since $M_t$ converges to $p$ in
Hausdorff distance, there exists $t_0\in(0,T)$ such that $X_{n+1}<0$ on $M_t$  for every $t\in[t_0,T)$.

We first establish two sign properties for the original hypersurface. 
On $\partial M_t$, the tangent hyperplane of the curved face is a
supporting hyperplane of the convex body $\Omega_t$. Since $p=-E$
belongs to $\Omega_t$, we have
\begin{equation*}
 \langle -E-X,\nu\rangle\leq0.
\end{equation*}
The free-boundary condition gives $\langle X,\nu\rangle=0$ on
$\partial M_t$. Hence
\begin{equation*}
 \langle\nu,E\rangle\geq0
 \qquad\text{on }\partial M_t.
\end{equation*}

We claim that $\langle\nu,E\rangle\geq0$ on the whole hypersurface. Otherwise, $\langle\nu,E\rangle$
would attain a negative minimum at an interior point. At this point,
\begin{equation*}
 0=\nabla_i \langle\nu,E\rangle
   =\langle D_{e_i}\nu,E\rangle
   =h_i{}^j\langle e_j,E\rangle.
\end{equation*}
Since $h_i{}^j$ is positive definite, the tangential component of
$E$ vanishes. It follows that $\nu=-E$. The supporting hyperplane
inequality with $p=-E$ then gives
\begin{equation*}
 0\geq\langle -E-X,\nu\rangle
   =\langle -E-X,-E\rangle
   =1+X_{n+1}.
\end{equation*}
This is impossible because an interior point of $M_t$ lies in the
open unit ball and therefore satisfies $X_{n+1}>-1$. Consequently,
\begin{equation}\label{s4.nuE}
 \langle\nu,E\rangle\geq0
 \qquad\text{on }M_t.
\end{equation}

We next show that
\begin{equation}\label{s4.Xnu}
 \langle X,\nu\rangle\leq0
 \qquad\text{on }M_t.
\end{equation}
The function $\langle X,\nu\rangle$ vanishes on $\partial M_t$.
If it had a positive maximum at an interior point, then
\begin{equation*}
 0=\nabla_i \langle X,\nu\rangle
   =h_i{}^j\langle X,e_j\rangle.
\end{equation*}
Strict convexity implies that $X^T=0$ at this point. Hence $ X=\langle X,\nu\rangle\nu$.  
Using $\langle X,\nu\rangle>0$ and \eqref{s4.nuE}, we obtain
\begin{equation*}
 X_{n+1}=\langle X,\nu\rangle\langle\nu,E\rangle\geq0,
\end{equation*}
which contradicts $X_{n+1}<0$. This proves \eqref{s4.Xnu}.

For the Cayley map \eqref{eq:cayley-alpha}, its differential has the form
\begin{equation*}
 dF_z=e^{w(z)}Q(z),
 \qquad Q(z)\in O(n+1).
\end{equation*}
A direct differentiation gives
\begin{equation}\label{s4.QzE}
 Q(z)(z+E)
  =
 \frac{\Lambda(z)}{2}\bigl(E-F(z)\bigr).
\end{equation}
The map $F$ preserves the interior side, so the unit normals satisfy $ \nu=Q(Z)N$.  
Using \eqref{s4.QzE}, $X=F(Z)$, \eqref{s4.nuE} and \eqref{s4.Xnu}, we find
\begin{align*}
 \langle Z+E,N\rangle
 &=
 \langle Q(Z)(Z+E),Q(Z)N\rangle \\
 &=
 \frac{\Lambda(Z)}{2}
 \langle E-X,\nu\rangle \\
 &=
 \frac{\Lambda(Z)}{2}
 \bigl(\langle E,\nu\rangle-\langle X,\nu\rangle\bigr)
 \geq0.
\end{align*}
It follows from \eqref{eq:Nw-alpha} that
\begin{equation}\label{s4.Nw}
 N(w)
 =
 -\frac{2\langle Z+E,N\rangle}{\Lambda(Z)}
 =
 -\langle E-X,\nu\rangle
 \leq0.
\end{equation}

By Lemma \ref{lem:kappa-lower-alpha}, the Weingarten map in the ball metric satisfies $\overline A\geq c_1\operatorname{Id}$.  
The conformal transformation formula \eqref{eq:conformal-A-alpha} and \eqref{s4.Nw}
give
\begin{align*}
 A_Z
 &=
 e^w\overline A-N(w)\operatorname{Id}  \geq
 c_1e^w\operatorname{Id}.
\end{align*}
Since $e^{-w}={\Lambda(Z)}/{2}\leq C_0$ 
on the relevant neighborhood, we obtain
\begin{equation*}
 A_Z\geq\frac{c_1}{C_0}\operatorname{Id}.
\end{equation*}
Thus \eqref{equ-AZlowbd} holds with $c_2=c_1/C_0$.
\end{proof}

For later use, we also observe that for $t\in [t_0,T)$ the Gauss image lies in the
upper hemisphere. A direct calculation from the Cayley map gives
\begin{equation*}
 Q(Z)E
 =
 \frac{2(1-X_{n+1})}{|E-X|^2}(E-X)-E.
\end{equation*}
Therefore
\begin{align*}
 \langle N,E\rangle
 &=
 \langle\nu,Q(Z)E\rangle \\
 &=
 \frac{
 (1-|X|^2)\langle\nu,E\rangle
 -
 2(1-X_{n+1})\langle X,\nu\rangle
 }{|E-X|^2}
 \geq0.
\end{align*}
Equality holds on the free boundary by \eqref{eq:flat-bd-alpha}. Hence
\begin{equation*}
 N(\Sigma_t)\subset\mathbb S^n_+,
 \qquad
 N(\partial\Sigma_t)\subset\partial\mathbb S^n_+.
\end{equation*}

Finally, the even reflection of $\Sigma_t$ across $\Pi$ is a
$C^2$ closed locally strictly convex hypersurface. The Hadamard
theorem implies that its Gauss map is a diffeomorphism onto
$\mathbb S^n$. By reflection symmetry, the restriction
\begin{equation*}
 N:\Sigma_t\longrightarrow\mathbb S^n_+
\end{equation*}
is a diffeomorphism. This also justifies the use of the Gauss map
coordinates in the following subsection.

\subsection{Volume normalization}
\label{sec:normalization-alpha}

Let $\Omega_t^{\HH}$ be the Euclidean region in $\HH$ bounded by $\Sigma_t$
and its flat part on $\PiH$. Let
\begin{equation*}
        V_+=|\B^{n+1}\cap\HH|=\frac{|\Sph^n_+|}{n+1}.
\end{equation*}
We normalize the flow \eqref{s4.flow-pb} so that the enclosed volume is preserved.  Define the normalized hypersurface by
\begin{equation*}
        Y(\cdot,s)=\mathrm{e}^s Z(\cdot,t),
\end{equation*}
where the normalized time parameter $s$ is given by
\begin{equation}
    \label{equ-s}
	s = \dfrac{1}{n+1}\log \dfrac{V_+}{|\Omega_t^{\mathbb{H}}|}.
\end{equation}
Since $\Omega_t$ contracts to $p=-E$ and $F^{-1}(p)=0$, the half-space volume $|\Omega_t^{\mathbb{H}}|$ tends to zero as $t\nearrow T$, and hence $s\to\infty$.  We also write $\rho(s)=\mathrm{e}^{-s}$.  By \eqref{equ-s}, the normalized region $\widehat\Omega_s=e^s\Omega_t^{\mathbb{H}}$ enclosed by $Y(M,s)$ has volume $V_+$.

For $Y$, let $A$ and $K=\det A$ denote the Euclidean Weingarten map and Gauss
curvature.  Since the Weingarten map of $Z=\rho Y$ is $\rho^{-1}A$, define
\begin{equation*}
	\Lambda = \Lambda(\rho Y) = 1+2\rho Y_{n+1}+\rho^2|Y|^2,
	\qquad
	q=-\frac{2\ip{E+\rho Y}{N}}{\Lambda}.
\end{equation*}
Set
\begin{equation*}
	\Theta=\Lambda^{n\alpha+1}
	\det(\Id+\rho q A^{-1})^\alpha.
\end{equation*}
Substituting $Z=\rho Y$ and $A_Z=\rho^{-1}A$ into
\eqref{eq:Phi-alpha-explicit} gives
\begin{equation*}
	\begin{aligned}
		\Phi_\alpha(\rho Y,N,\rho^{-1}A)
		&=2^{-(n\alpha+1)}\Lambda^{n\alpha+1}
		\det(\rho^{-1}A+q \Id)^\alpha  \\
		&= 2^{-(n\alpha+1)}\rho^{-n\alpha}\Theta K^{\alpha}.
	\end{aligned}
\end{equation*}
The relation between the time parameters $t$ and $s$ is given by
\begin{align*}
	\dfrac{\mathrm{d}s}{\mathrm{d}t} = & \dfrac{\mathrm{e}^{(n+1)s}}{|\mathbb{S}_+^n|}\int_{Z(M,t)} \Phi_{\alpha}(Z,N,A_Z)\mathrm{d}\mu_t^{\mathbb{H}}\\
	= & \dfrac{\mathrm{e}^{s}}{|\mathbb{S}_+^n|} \int_{Y(M,s)} \Phi_{\alpha}(\rho Y, N, \rho^{-1}A) \mathrm{d}\mu_s\\
	=& \dfrac{\mathrm{e}^s}{|\mathbb{S}_+^n|}\int_{\mathbb{S}_+^n} 2^{-(n\alpha+1)}\rho^{-n\alpha}\Theta K^{\alpha-1} \mathrm{d}\sigma.
\end{align*}
Differentiation gives
\begin{align*}
	\partial_s Y(x,s) &= Y(x,s) + \mathrm{e}^{s}  \partial_{t} Z(x,t)  \dfrac{\mathrm{d}t}{\mathrm{d}s}\\
	&= Y(x,s) - \dfrac{\Theta K^{\alpha}}{\dfrac{1}{|\mathbb{S}_+^n|}\displaystyle\int_{\mathbb{S}_+^n}\Theta K^{\alpha-1}\mathrm{d}\sigma}N.
\end{align*}

By Lemma \ref{lem:large-s-gauss-alpha}, after increasing $s_0=s_0(t_0)$ if necessary,
we use the Gauss map variable
\begin{equation*}
	\xi=N\in\Sph^n_+=\{\xi\in\Sph^n:\xi_{n+1}\ge0\}.
\end{equation*}
The support function is
\begin{equation}\label{s5.supp-h}
	h(\xi,s)=\langle Y(\xi,s),\xi \rangle,
	\qquad
	Y(\xi,s)=\overline{\nabla} h(\xi,s)+h(\xi,s)\xi.
\end{equation}
Let
\begin{equation*}
	B_{ij}=\overline{\nabla}_i\overline{\nabla}_j h+h\sigma_{ij},
	\qquad
	A=B^{-1},
	\qquad
	K=(\det B)^{-1}.
\end{equation*}
In the support-function coordinates \eqref{s5.supp-h}, the quantities $\Lambda$, $q$, and $\Theta$ are given by
\begin{equation}\label{s5.Dq}
	\Lambda= 1 + 2\rho \langle \overline{\nabla} h+h\xi, E \rangle + \rho^2(h^2+|\overline{\nabla} h|^2_{\sigma}),
	\qquad
	q= -\frac{2(\xi_{n+1}+\rho h)}{\Lambda},
\end{equation}
and
\begin{equation*}
	\Theta=\Lambda^{n\alpha+1}\det(\Id+\rho q B)^\alpha.
\end{equation*}
Let $\eta$ be the outward conormal of the equator in $\Sph^n_+$.  A standard computation shows that the support function of the normalized conformal flow satisfies
\begin{equation}\label{eq:support-flow-alpha}
	\left\{
	\begin{array}{ll}
		\partial_sh=h-\dfrac{\Theta K^{\alpha}}{\dfrac{1}{|\mathbb{S}_+^n|}\displaystyle\int_{\mathbb{S}_+^n}\Theta K^{\alpha-1}\mathrm{d}\sigma}, & \hbox{in }\Sph^n_+,\\
		\overline{\nabla}_\eta h=0, & \hbox{on }\bd\Sph^n_+.
	\end{array}
	\right.
\end{equation}

Since $p=-E\in \Omega_t$ and $F^{-1}(p)=0$, the origin belongs to the half-body $\widehat{\Omega}_s$.  Consequently, $h(\xi,s)\geq 0,\ \xi\in\mathbb{S}^n_+$.

\section{Half-space entropy and support-function bounds}
\label{sec:entropy-alpha}
Our goal of this section is to prove uniform positive upper and lower bounds for the support function $h$ along the normalized flow \eqref{eq:support-flow-alpha}. We use the half-space entropy as the main tool. Its almost monotonicity, with error controlled by the integrable diameter decay, gives the entropy limit. A two-time comparison then implies convergence of the entropy point to the origin. Together with the static entropy geometry, this yields uniform bounds on the support function.


Recall that the volume normalization is equivalent to
\begin{equation}\label{eq:half-minkowski-normalized}
        \dfrac{1}{|\mathbb{S}_+^n|}\int_{\mathbb{S}_+^n}\frac{h}{K}\mathrm{d}\sigma=1,
\end{equation}
since the left hand side equals $(n+1)|\widehat\Omega_s|/|\mathbb{S}^n_+|$.

\subsection{Half-space entropy geometry}
We first recall the entropy facts for convex half-bodies. For $z\in\PiH$ denote
\begin{equation*}
        u_z(\xi)=h(\xi)-\langle z,\xi\rangle,\qquad \xi\in\mathbb{S}^n_+.
\end{equation*}
A point $z$ is admissible if $u_z>0$.  For $\alpha\ne1$ define
\begin{equation*}
        \cE_{\alpha,+}(\Omega,z)
        =\frac{\alpha}{\alpha-1}
        \log\left(\dfrac{1}{|\mathbb{S}_+^n|}\int_{\mathbb{S}_+^n} u_z^{1-\frac{1}{\alpha}}\mathrm{d}\sigma\right),
\end{equation*}
and for $\alpha=1$ define
\begin{equation*}
        \cE_{1,+}(\Omega,z)=\dfrac{1}{|\mathbb{S}_+^n|}\int_{\mathbb{S}_+^n}\log u_z \mathrm{d}\sigma.
\end{equation*}
The entropy of the half-body is
\begin{equation*}
        \cE_{\alpha,+}(\Omega)
        =\sup_{z\in\PiH,\ u_z>0}\cE_{\alpha,+}(\Omega,z).
\end{equation*}

Let $\Omega^d$ denote the double body obtained by even reflection of $\Omega$ across $\PiH$.  It is a
closed convex body, $|\Omega^d|=2|\Omega|$, and its support function is the even
extension of $h$.  Since $|\widehat\Omega_s|=V_+$, the doubled normalized body
has the volume of the Euclidean unit ball.  The entropy point of $\Omega^d$ is
unique by \cite[Lemma 2.5]{AGN2016}. Reflection symmetry forces it to lie in
$\PiH$.  Consequently the half-space entropy is exactly the closed entropy of
the doubled body.

\begin{proposition}\label{prop:half-entropy-geometry}
Let $\alpha > \frac{1}{n+2}$ and let $\Omega\subset\HH$ be a compact convex half-body
with nonempty interior and $|\Omega|=V_+$.  Then the following statements hold.
\begin{enumerate}[label=\rm(\roman*)]
\item There is a unique entropy point $z_e(\Omega)\in\Int(\Omega^d)\cap\PiH$.  Equivalently, $z_e(\Omega)$ lies in the relative interior of the flat part $\Omega\cap\PiH$.
It is characterized by
\begin{equation}\label{eq:entropy-point-condition}
        \int_{\Sph^n_+}\xi_\beta u_{z_e}^{-1/\alpha}\,\dd\sigma=0,
        \qquad \beta=1,\ldots,n.
\end{equation}
\item The entropy is nonnegative: $ \cE_{\alpha,+}(\Omega)\ge0$.  
\item There are constants $b=b(n,\alpha)>0$ and $C=C(n,\alpha)$ such that
\begin{equation}\label{eq:entropy-radius-control}
\begin{split}
        \min\{r_-(\Omega),\omega_-(\Omega)\}
        &\ge C^{-1}e^{-b\cE_{\alpha,+}(\Omega)},\\
        \max\{r_+(\Omega),\omega_+(\Omega)\}
        &\le Ce^{nb\cE_{\alpha,+}(\Omega)},
\end{split}
\end{equation}
where $\omega_-$ and $\omega_+$ denote the minimum and maximum width of $\Omega$, respectively. 
\item For every $L<\infty$ there are $d_L,c_L>0$ such that, if
$\cE_{\alpha,+}(\Omega)\le L$, then
\begin{equation}\label{eq:entropy-point-interior}
        \dist(z_e(\Omega),\partial\Omega^d)\ge d_L
\end{equation}
and, for every admissible $z\in\PiH$,
\begin{equation}\label{eq:entropy-quadratic-stability}
        \cE_{\alpha,+}(\Omega,z)
        \le \cE_{\alpha,+}(\Omega)
        -c_L\min\{1,|z-z_e(\Omega)|^2\}.
\end{equation}
\end{enumerate}
\end{proposition}

\begin{proof}
Apply Corollary 2.2, Proposition 2.7, and Lemmas 4.2--4.4 of
\cite{AGN2016} to the doubled body $\Omega^d$.  The volume of $\Omega^d$ is
$|\B^{n+1}|$, and its entropy point belongs to the fixed hyperplane of the
reflection.  Restriction to the upper hemisphere gives
\eqref{eq:entropy-point-condition}.  The inradius, circumradius and width of
$\Omega^d$ are comparable, with dimensional constants, to the corresponding
quantities of the half-body.  This gives \eqref{eq:entropy-radius-control}.
The compactness argument in \cite[Lemmas 4.2--4.4]{AGN2016} gives
\eqref{eq:entropy-point-interior} and
\eqref{eq:entropy-quadratic-stability}.
\end{proof}

\subsection{Decay of the unnormalized diameter}
\label{sub5.2}
Set
\begin{equation}\label{eq:delta-def}
        \delta(s)=\sup_{\Sph^n_+}|\rho(s)Y(\cdot,s)|.
\end{equation}
Thus $\delta(s)$ is controlled by the Euclidean circumradius of the
unnormalized half-space body. We also have
\begin{equation*}
    \rho(s)\leq C\delta(s).
\end{equation*}
Indeed, by the volume normalization $|\widehat\Omega_s|=V_+$ and  $ \widehat\Omega_s
 \subset  B_{\sup_{\mathbb S^n_+}|Y(\cdot,s)|}(0)$, 
we have
\begin{equation*}
 V_+
 \leq
 |B^{n+1}|
 \left(\sup_{\mathbb S^n_+}|Y(\cdot,s)|\right)^{n+1}.
\end{equation*}
It follows that $ \sup_{\mathbb S^n_+}|Y(\cdot,s)|
 \geq c(n)>0$.  
By the definition \eqref{eq:delta-def}, 
we obtain $\rho(s)\leq C\delta(s)$.

The next estimate is the counterpart of
\cite[(4.11)]{CH2024}.

\begin{lemma}\label{lem:delta-decay-alpha}
There is $C<\infty$, depending only on $n,\alpha$ and the initial hypersurface,
such that
\begin{equation}\label{eq:delta-decay-alpha}
        \delta(s)\le C\rho(s)^{\frac{n+1}{2n+1}}
        =C\mathrm{e}^{-\frac{n+1}{2n+1}s}
\end{equation}
for every sufficiently large $s$.  In particular,
$\delta\in L^1([s_0,\infty))$.
\end{lemma}

\begin{proof}
Reflect the unnormalized half-body across $\PiH$ and denote its inradius and
circumradius by $r_-$ and $r_+$.  The even reflection is a closed $C^2$ strictly convex hypersurface.
If necessary, we approximate its even support function in $C^2$ by
smooth strictly convex even support functions. The lower bound for
the principal curvatures is preserved, up to decreasing the constant.
Applying \cite[Lemma~2.2]{CW00} to the approximating bodies and
passing to the limit gives
\begin{equation}\label{eq:radius-square-decay}
        r_+^2\le Cr_-.
\end{equation}
For every convex body,
\begin{equation*}
        c_n r_-^n r_+\le |\Omega^d_s|\le C_n r_-r_+^n.
\end{equation*}
Since $|\Omega^d_s|=2V_+\rho^{n+1}$, the first inequality and
\eqref{eq:radius-square-decay} imply
\begin{equation*}
        c r_+^{2n+1}\le 2V_+\rho^{n+1}.
\end{equation*}
The origin belongs to the unnormalized body, so every boundary point has
length at most twice the circumradius. This proves
\eqref{eq:delta-decay-alpha}.
\end{proof}

\subsection{Entropy bounds for $\frac{1}{n+2}<\alpha<1$}
In this subsection, we prove a uniform entropy bound for the range $\frac{1}{n+2}<\alpha<1$. The argument relies on a free boundary Minkowski--Reilly inequality and the monotonicity of a suitable Andrews-type quantity \cite{And00}.

We first state the main result of this subsection.
\begin{proposition}
	\label{prop:subunit-entropy-bound-alpha}
	Assume $\frac1{n+2}<\alpha<1$. Then there are constants $s_*<\infty$ and $C<\infty$, depending only on $n,\alpha$ and the initial hypersurface, such that for $s\ge s_*$,
	\begin{equation}\label{eq:subunit-entropy-bound-alpha}
		\mathcal{E}_{\alpha,+}(\widehat\Omega_s)\le C.
	\end{equation}
\end{proposition}

To prove this proposition, we need a free boundary analogue of the classical Minkowski--Reilly inequality.
\begin{lemma}
	\label{lem:FB-Minkowski-Reilly}
    Let $M$ be a strictly convex free-boundary hypersurface in $\overline{\B}^{n+1}$ whose boundary satisfies $\partial M\subset\mathbb{S}^n$.  Let $\Omega$ be the domain enclosed by $M$ and $\mathbb{S}^n$, and let $\Gamma\subset\mathbb{S}^{n}$ be the domain enclosed by $\partial M$.
    Let $h_{ij}$ be the second fundamental form of $M$ with respect to the outward normal $\nu$, let $b^{ij}$ be its
	inverse, and let $H=\tr h$. Then, for every smooth function $\varphi$ on $M$,
	\begin{equation}\label{eq:FB-Minkowski-Reilly}
		\int_M H\varphi^2\,\dd \mu
		-\int_M b^{ij}\nabla_i\varphi\nabla_j\varphi\,\dd \mu
		\le
		\frac n{n+1}
		\frac{\left(\int_M\varphi\,\dd \mu\right)^2}{|\Omega|} .
	\end{equation}
\end{lemma}

\begin{proof}
Set
\begin{equation*}
    c=\frac{1}{|\Omega|}\int_M\varphi\,\dd\mu.
\end{equation*}
Consider the Neumann problem
\begin{equation}\label{eq:MR-Neumann-problem}
    \begin{cases}
        \Delta u=c, & \text{in }\Omega,\\
        \partial_\nu u=\varphi, & \text{on }M,\\
        \partial_{\nu'}u=0, & \text{on }\Gamma,
    \end{cases}
\end{equation}
where $\nu$ is the outward unit normal of $M$ and
$\nu'=X$ is the outward unit normal of the unit sphere.
The definition of $c$ gives the compatibility condition
\begin{equation*}
    c|\Omega|
    =
    \int_M\varphi\,\dd\mu
    =
    \int_{\partial\Omega}\partial_{\nu_{\partial\Omega}}u
    \,\dd\mu_{\partial\Omega}.
\end{equation*}
Thus \eqref{eq:MR-Neumann-problem} has a weak solution, unique up
to an additive constant. We normalize it by
\begin{equation*}
    \int_\Omega u\,\dd x=0.
\end{equation*}

The boundary of $\Omega$ consists of two smooth faces meeting
orthogonally along the edge $\partial M$. In
local smooth coordinates, the model domain near the edge is a
dihedral domain with opening angle $\pi/2$. Since the Neumann
data $\varphi$ on $M$ and $0$ on $\Gamma$ are smooth and satisfy
the compatibility condition above, the regularity theory for the
Neumann problem on domains with codimension-two edges gives $u\in H^2(\Omega)$. 
See \cite[Section~23.C, Theorem~23.3, and
Example~23.11(1), pp.~196--198]{Dauge1988}.
This regularity is sufficient to justify the integrations by
parts below, for instance by approximation in $H^2(\Omega)$.

We now derive the Reilly identity, retaining the edge
contribution explicitly. Let
\begin{equation*}
    V^i=(\Delta u)u_i-u_{ij}u_j.
\end{equation*}
A direct calculation gives
\begin{equation*}
    \partial_iV^i=(\Delta u)^2-|D^2u|^2.
\end{equation*}
Let $\mu$ and $\mu'$ denote the outward unit conormals of
$\partial M$ in $M$ and $\Gamma$, respectively. Applying the
divergence theorem to $V$ and then integrating tangentially on
each smooth face gives
\begin{align}
    \int_\Omega\big((\Delta u)^2-|D^2u|^2\big)\,\dd x
    ={}&
    \int_M
    \left(
        Hu_\nu^2
        +2u_\nu\Delta_Mu
        +h(\nabla^Mu,\nabla^Mu)
    \right)\,\dd\mu
    \nonumber\\
    &+
    \int_\Gamma
    \left(
        H_\Gamma u_{\nu'}^2
        +2u_{\nu'}\Delta_\Gamma u
        +h_\Gamma(\nabla^\Gamma u,\nabla^\Gamma u)
    \right)\,\dd\mu_\Gamma
    \nonumber\\
    &-
    \int_{\partial M}
    \left(
        u_\nu u_\mu
        +u_{\nu'}u_{\mu'}
    \right)\,\dd S.
    \label{eq:Reilly-with-edge}
\end{align}
Indeed, the last line is precisely the boundary contribution
produced when the tangential integration by parts is performed
separately on $M$ and $\Gamma$.

We now show that the edge term in
\eqref{eq:Reilly-with-edge} vanishes. Along $\partial M$, the
free-boundary condition implies $\mu=\nu'$.  
Since $u_{\nu'}=0$ on $\Gamma$, we have
\begin{equation*}
    u_\mu=u_{\nu'}=0
    \qquad\text{on }\partial M.
\end{equation*}
Therefore the first term in the edge integral vanishes. The
second term also vanishes because $u_{\nu'}=0$. Hence there is no
corner contribution.

On the spherical face $\Gamma$, we have $ h_\Gamma=g_\Gamma$ and $H_\Gamma=n$. 
Since $u_{\nu'}=0$, its contribution to
\eqref{eq:Reilly-with-edge} is
\begin{equation*}
    \int_\Gamma|\nabla^\Gamma u|^2\,\dd\mu_\Gamma\geq0.
\end{equation*}
Using $u_\nu=\varphi$ on $M$, we consequently obtain
\begin{align}
    \int_\Omega\big((\Delta u)^2-|D^2u|^2\big)\,\dd x
    ={}&
    \int_M
    \left(
        H\varphi^2
        +2\varphi\Delta_Mu
        +h(\nabla^Mu,\nabla^Mu)
    \right)\,\dd\mu
    \nonumber\\
    &+
    \int_\Gamma|\nabla^\Gamma u|^2\,\dd\mu_\Gamma
    \nonumber\\
    \geq{}&
    \int_M
    \left(
        H\varphi^2
        +2\varphi\Delta_Mu
        +h(\nabla^Mu,\nabla^Mu)
    \right)\,\dd\mu.
    \label{eq:Reilly-piecewise}
\end{align}

We next integrate by parts on $M$. Since $\mu=\nu'$ along
$\partial M$ and $u_{\nu'}=0$ on $\Gamma$, we have
$u_\mu=0$ along $\partial M$. It follows that
\begin{align*}
    2\int_M\varphi\Delta_Mu\,\dd\mu
    &=
    -2\int_M
    \langle\nabla\varphi,\nabla^Mu\rangle\,\dd\mu
    +
    2\int_{\partial M}\varphi u_\mu\,\dd S\\
    &=
    -2\int_M
    \langle\nabla\varphi,\nabla^Mu\rangle\,\dd\mu.
\end{align*}
Completing the square with respect to the positive definite
second fundamental form $h$ gives
\begin{align*}
    &h(\nabla^Mu,\nabla^Mu)
    -2\langle\nabla\varphi,\nabla^Mu\rangle\\
    &\qquad=
    -b^{ij}\nabla_i\varphi\nabla_j\varphi\\
    &\qquad\quad+
    h_{ij}
    \left(
        \nabla^iu-b^{ik}\nabla_k\varphi
    \right)
    \left(
        \nabla^ju-b^{j\ell}\nabla_\ell\varphi
    \right)\\
    &\qquad\geq
    -b^{ij}\nabla_i\varphi\nabla_j\varphi.
\end{align*}
Substituting this estimate into
\eqref{eq:Reilly-piecewise}, we obtain
\begin{equation}\label{eq:Reilly-lower-bound}
    \int_\Omega
    \big((\Delta u)^2-|D^2u|^2\big)\,\dd x
    \geq
    \int_MH\varphi^2\,\dd\mu
    -
    \int_M
    b^{ij}\nabla_i\varphi\nabla_j\varphi\,\dd\mu.
\end{equation}

On the other hand, the pointwise inequality
\begin{equation*}
    |D^2u|^2\geq\frac{1}{n+1}(\Delta u)^2
\end{equation*}
and the equation $\Delta u=c$ imply
\begin{align*}
    \int_\Omega
    \big((\Delta u)^2-|D^2u|^2\big)\,\dd x
    &\leq
    \frac{n}{n+1}
    \int_\Omega(\Delta u)^2\,\dd x\\
    &=
    \frac{n}{n+1}c^2|\Omega|\\
    &=
    \frac{n}{n+1}
    \frac{
        \left(\int_M\varphi\,\dd\mu\right)^2
    }{|\Omega|}.
\end{align*}
Combining this inequality with
\eqref{eq:Reilly-lower-bound} proves
\eqref{eq:FB-Minkowski-Reilly}.
\end{proof}

With the above inequality, we show that the Andrews-type quantity \cite[\S 4]{And00} is nonincreasing for $0<\alpha<1$.

\begin{lemma}	\label{prop:FB-Z-monotonicity}
	Let $M_t$ be a smooth strictly convex free-boundary solution of
	the flow \eqref{eq:flow} in the unit ball, and set
	\begin{equation*}
		I_\alpha(t)=\int_{M_t}K^\alpha\,\dd \mu_t .
	\end{equation*}
	If $0<\alpha<1$, then
	\begin{equation}\label{s5.And}
		\mathcal Z_\alpha(t)
		=|\Omega_t|^{\frac n{n+1}}I_\alpha(t)^{\frac1{\alpha-1}}
	\end{equation}
	is nonincreasing.
\end{lemma}

\begin{proof}
	Since the spherical part of $\partial\Omega_t$ moves tangentially to the support sphere, only $M_t$ contributes to the first variation of volume, and hence
	\begin{equation}\label{eq:FB-volume-variation-subunit}
		\dfrac{\mathrm{d}}{\mathrm{d}t}|\Omega_t|=-\int_{M_t} K^\alpha \,\dd \mu_t=-I_\alpha(t).
	\end{equation}
	Along the flow \eqref{eq:flow}, using equations \eqref{eq:evol-G} and \eqref{eq:dmu} gives
	\begin{equation}\label{eq:I-alpha-variation-before-ibp}
		\frac{\dd}{\dd t}I_\alpha(t)
		=\alpha\int_{M_t}K^\alpha b^{ij}\nabla_i\nabla_j G\,\dd \mu_t
		+(\alpha-1)\int_{M_t}H G^2\,\dd \mu_t .
	\end{equation}
	Since $K b^{ij}=\dot{K}^{ij}$ is divergence-free, integration by parts gives
	\begin{equation}\label{eq:Kalpha-second-term-ibp}
		\begin{split}
			\alpha\int_{M_t} K^\alpha b^{ij}\nabla_i\nabla_j G\,\dd \mu_t
			&=-(\alpha-1)\int_{M_t} b^{ij}\nabla_i G\nabla_j G\,\dd \mu_t  \\
			&\quad +\alpha\int_{\partial M_t}K^\alpha b^{\mu j}\nabla_j G\,\dd s_t.
		\end{split}
	\end{equation}
	Along the free boundary $\partial M_t$, by Lemma \ref{lem:bd-identities-alpha}   $h_{\mu\beta}=0$ for directions tangent to $\partial M_t$ and $\nabla_\mu G=G$.  Therefore
	\begin{equation}\label{eq:FB-Z-boundary-positive}
		K^\alpha b^{\mu j}\nabla_j G
		=G b^{\mu\mu}\nabla_\mu G
		=\frac{G^2}{h_{\mu\mu}},
	\end{equation}
	and the boundary term in \eqref{eq:Kalpha-second-term-ibp} is nonnegative.  Combining \eqref{eq:I-alpha-variation-before-ibp}--\eqref{eq:FB-Z-boundary-positive}, 	we obtain 
	\begin{align*}
		\dfrac{\mathrm{d}}{\mathrm{d}t}I_\alpha(t)
		=&-(1-\alpha)\left(
		\int_{M_t}H G^2\,\dd \mu_t
		-\int_{M_t}b^{ij}\nabla_i G\nabla_j G\,\dd \mu_t\right)\\
		&\qquad +\alpha\int_{\partial M_t}\frac{G^2}{h_{\mu\mu}}\,\dd s_t.
	\end{align*}
	Applying Lemma \ref{lem:FB-Minkowski-Reilly} with $\varphi=G$ gives
	\begin{equation*}
		\dfrac{\mathrm{d}}{\mathrm{d}t}I_\alpha(t)
		\ge -(1-\alpha)\frac n{n+1}\frac{I_\alpha(t)^2}{|\Omega_t|}.
	\end{equation*}
	Together with \eqref{eq:FB-volume-variation-subunit}, this implies
	\begin{align*}
		\frac{\dd}{\dd t}\log\mathcal Z_\alpha(t)
		=& \frac n{n+1}\frac{\frac{d}{\mathrm{d}t}|\Omega_t|}{|\Omega_t|}
		+\frac1{\alpha-1}\frac{\frac{\mathrm{d}}{\mathrm{d}t}I_{\alpha}(t)}{I_\alpha(t)} \\
		=&-\frac n{n+1}\frac{I_\alpha(t)}{|\Omega_t|}
		+\frac1{\alpha-1}\frac{\frac{\mathrm{d}}{\mathrm{d}t}I_\alpha(t)}{I_\alpha(t)} \leq 0,
	\end{align*}
	and the assertion follows.
\end{proof}

We are now in a position to prove the entropy bound.

\begin{proof}[Proof of Proposition \ref{prop:subunit-entropy-bound-alpha}]
 By Lemma \ref{prop:FB-Z-monotonicity}, after fixing any sufficiently late time $t_0<T$,
	\begin{equation*}
		\mathcal Z_\alpha(t)
		\le \mathcal Z_\alpha(t_0),
		\qquad t\in[t_0,T).
	\end{equation*}
	
	Let $(\Omega_t^{\HH})^d$ be the double body obtained by even reflection of the unnormalized half-space body across $\PiH$.  Since $M_t$ converges to the extinction point $p=-E$ and $F^{-1}$ has nonsingular differential near $p$, increasing $t_0$ if necessary gives the volume comparison
	\begin{equation}\label{eq:subunit-volume-comparison}
		c|\Omega_t|\le |(\Omega_t^{\HH})^d|\le C|\Omega_t|.
	\end{equation}
	Recall the notations in Section \ref{sec:normalization}.  We claim that we have the curvature comparison
	\begin{equation}\label{eq:AZ-barA-comparison-subunit}
		c\bar A\le A_Z\le C\bar A
		\qquad\hbox{on }\Sigma_t,
		\quad t\in[t_0,T).
	\end{equation}
In fact, by Lemma \ref{lem:large-s-gauss-alpha}, we have $N(w)\leq0$. The conformal transformation
formula \eqref{eq:conformal-A-alpha} can be rewritten as
\begin{equation*}
	A_Z=e^w\bar A-N(w)\Id.
\end{equation*}
Since $e^w$ is uniformly bounded from above and below on the relevant
neighborhood of the origin, it follows that $A_Z\geq e^w\bar A\geq c\bar A$.  
On the other hand, $|N(w)|\leq C$ and Lemma \ref{lem:kappa-lower-alpha} gives
$\bar A\geq c_1\Id$. Hence
\begin{equation*}
	A_Z
	\leq e^w\bar A+|N(w)|\Id
	\leq C\bar A.
\end{equation*}
This proves \eqref{eq:AZ-barA-comparison-subunit}. Since the area
elements of the Euclidean and conformal metrics are also uniformly
comparable, \eqref{eq:AZ-barA-comparison-subunit} gives
	\begin{equation}\label{eq:subunit-curvature-integral-comparison}
		cI_\alpha(t)
		\le \int_{\Sigma_t}K_Z^\alpha\,\dd \mu_Z
		\le CI_\alpha(t).
	\end{equation}
	The reflected half body contributes the same integral. Therefore, the volume and curvature comparisons \eqref{eq:subunit-volume-comparison} and
	\eqref{eq:subunit-curvature-integral-comparison} imply that the Andrews-type quantity \eqref{s5.And} of $(\Omega_t^{\HH})^d$ is bounded above for all $ t\in[t_0,T)$.  
	Since $\alpha>\frac{1}{n+2}$, after passing through a smooth even approximation, Andrews' estimate \cite[Theorem 4]{And00} therefore yields 
	\begin{equation}\label{eq:subunit-radius-ratio-double}
		\frac{r_+((\Omega_t^{\HH})^d)}{r_-((\Omega_t^{\HH})^d)}\le  C,
		\qquad t\in[t_0,T).
	\end{equation}
	Now set
	\begin{equation*}
		(\widehat\Omega_s)^d=\rho(s)^{-1}(\Omega_t^{\HH})^d,
		\qquad s=s(t).
	\end{equation*}
	The ratio \eqref{eq:subunit-radius-ratio-double} is scale invariant, and
	$|(\widehat \Omega_s)^d|=2V_+=|\B^{n+1}|$.  A fixed volume together with a fixed inradius--circumradius ratio gives
	\begin{equation*}
		\operatorname{diam} (\widehat \Omega_s)^d\le R_*,
		\qquad s\ge s_*:=s(t_0),
	\end{equation*}
	for some finite $R_*$ depending only on $n,\alpha$ and the initial
	hypersurface.  Replacing $R_*$ by $\max\{R_*,1\}$, we assume $R_*\ge1$.
	
	The diameter bound directly controls the entropy. Let $z\in\PiH$ be any admissible center for the half-body $\widehat\Omega_s$, and write
	\begin{equation*}
		u_z(\xi)=h(\xi,s)-\langle z,\xi\rangle .
	\end{equation*}
	The support function of the double $(\widehat \Omega_s)^d$ is obtained by even reflection. Since $z\in\PiH$, the condition $u_z>0$ on $\Sph^n_+$ implies
	that $z\in\Int (\widehat \Omega_s)^d$.
    Therefore, for every $\xi\in\Sph^n_+$,
	\begin{equation}\label{eq:subunit-support-diameter-control}
		0<u_z(\xi)
		=\sup_{x\in\widehat\Omega_s}\langle x-z,\xi\rangle
		\le \operatorname{diam} (\widehat \Omega_s)^d
		\le R_* .
	\end{equation}
	Since $1-\frac1\alpha<0$, \eqref{eq:subunit-support-diameter-control} yields
	\begin{equation*}
		\begin{split}
			\cE_{\alpha,+}(\widehat\Omega_s,z)
			&=\frac{\alpha}{\alpha-1}\log\left( \dfrac{1}{|\mathbb{S}_+^n|}\int_{\mathbb{S}_+^n} u_z^{1-\frac1\alpha} \mathrm{d}\sigma  \right)  \\
			&\le \frac{\alpha}{\alpha-1}\log(R_*^{1-\frac1\alpha})
			=\log R_* .
		\end{split}
	\end{equation*}
	This bound is uniform in the admissible center $z$.  Taking the supremum over all admissible $z\in\PiH$ proves \eqref{eq:subunit-entropy-bound-alpha}.
\end{proof}

The entropy bound and the static entropy geometry give the following control on the support function relative to the entropy point.

\begin{corollary}
	\label{cor-bounded}
	Assume $\frac{1}{n+2} < \alpha <1$.  Then there are constants $s_*<\infty$ and $C<\infty$, depending only on
	$n,\alpha$ and the initial hypersurface, such that
	\begin{equation*}
		\frac{1}{C} \leq u_{z_e(s)}\leq C , \quad |\overline{\nabla}u_{z_e(s)}|_{\sigma} \leq C,
		\qquad s\ge s_*.
	\end{equation*}
    Here $z_e$ is the entropy point, and $u_{z_e}=h-\langle z_e,\xi \rangle$.
\end{corollary}

\begin{proof}
By Proposition \ref{prop:subunit-entropy-bound-alpha}, the entropy of the normalized half-body is
uniformly bounded from above. Hence Proposition \ref{prop:half-entropy-geometry}, applied to
the doubled body, gives constants $d_0>0$ and $R_0<\infty$ such
that
\begin{equation*}
    \dist\bigl(z_e(s),\partial(\widehat\Omega_s)^d\bigr)
    \geq d_0,
    \qquad
    \diam(\widehat\Omega_s)^d\leq R_0
\end{equation*}
for all sufficiently large $s$.

Since $u_{z_e}$ is the support function of
$(\widehat\Omega_s)^d-z_e(s)$ restricted to
$\mathbb S^n_+$, the first estimate implies
\begin{equation*}
    u_{z_e}(\xi,s)\geq d_0,
\end{equation*}
while the diameter bound gives
\begin{equation*}
    u_{z_e}(\xi,s)\leq R_0.
\end{equation*}
Moreover, in Gauss map coordinates,
\begin{equation*}
    Y-z_e(s)
    =
    \overline{\nabla} u_{z_e}+u_{z_e}\xi.
\end{equation*}
Since $z_e(s)\in(\widehat\Omega_s)^d$, we have
\begin{equation*}
    |Y-z_e(s)|\leq\diam(\widehat\Omega_s)^d\leq R_0.
\end{equation*}
Therefore,
\begin{equation*}
    |\overline{\nabla} u_{z_e}|_\sigma
    \leq |Y-z_e(s)|
    \leq R_0.
\end{equation*}
\end{proof}

The following lemma provides uniform $C^0$ and $C^1$ bounds for the transported functions on a unit time strip, which are needed in the range $\frac{1}{n+2}<\alpha<1$ of the almost monotonicity argument in Lemma \ref{lem:transported-entropy-alpha}.

\begin{lemma}
	\label{lem:unit-strip-transport-C0}
	Assume $\frac{1}{n+2} < \alpha <1$.  There are constants $c_0,C_0>0$ with the following property.  Let $S$ be large, for $S-1\le s\le S$,
	\begin{equation*}
		u^S(\xi,s)=h(\xi,s)-e^{s-S}\langle z_e(S), \xi \rangle.
	\end{equation*}
	Then
	\begin{equation*}
		c_0\le u^S(\xi,s)\le C_0,
		\qquad
		|\overline{\nabla} u^S(\xi,s)|\le C_0
	\end{equation*}
	for all $\xi\in\Sph^n_+$ and all $s\in[S-1,S]$.
\end{lemma}

\begin{proof}
	At the terminal time, $u^S(\xi,S)=h(\xi,S)-\langle z_e(S),\xi\rangle$.
	Hence Corollary \ref{cor-bounded} gives
	$u^S(\cdot,S)\ge c$. Since
	\begin{equation}
		\partial_s(e^{-s}u^S)
		=-e^{-s}\dfrac{\Theta K^{\alpha}}{\dfrac{1}{|\mathbb{S}_+^n|}\displaystyle\int_{\mathbb{S}_+^n}\Theta K^{\alpha-1}\mathrm{d}\sigma}\le0,
        \label{equ-uS}
	\end{equation}
	we have, for $S-1\le s\le S$,
	\begin{equation*}
		u^S(\xi,s)\ge e^{s-S}u^S(\xi,S)\ge e^{-1}c.
	\end{equation*}
	This proves the lower bound.

    Since $\diam{(\widehat{\Omega}_s)^d} \leq R_*$ and $0\in\widehat{\Omega}_s$, we know
    \begin{equation*}
        0 \leq h(\xi,s)\leq R_*,\qquad |Y(\xi,s)|\leq R_*.
    \end{equation*}
    Since $z_e(S)\in (\widehat{\Omega}_S)^d$ and $0\in\widehat{\Omega}_S$, we have $|z_e(S)|\leq R_*$. Therefore,
    \begin{align*}
        u^S = &h - \mathrm{e}^{s-S}\langle z_e(S),\xi \rangle \leq 2R_*,  \\
        |\overline{\nabla}h|_{\sigma} \leq & |Y|+h \leq 2R_*,\\
        |\overline{\nabla}u^S|_{\sigma} \leq & |\overline{\nabla}h|_{\sigma} + |z_e(S)| \leq 3R_*.
    \end{align*}
    This completes the proof.
\end{proof}

\subsection{Almost monotonicity for $\alpha>\frac{1}{n+2}$}
In this subsection we prove an almost monotonicity property of the entropy (Proposition~\ref{thm:almost-monotonicity-alpha}), which is crucial for obtaining the entropy limit, the convergence of the entropy point, and uniform bounds on the support function. The main difficulty arises from the conformal factor in the normalized flow. We will show that its effect can be controlled by the integrable diameter decay $\delta(s)$ from Section \ref{sub5.2}.

\begin{proposition}[Almost monotonicity]\label{thm:almost-monotonicity-alpha}
Let $\alpha>\frac{1}{n+2}$.  For all $S\ge s\ge s_1$ with $S-s\le1$,
\begin{equation}\label{eq:two-time-entropy-inequality}
\cE_{\alpha,+}(\widehat\Omega_{S})
-\cE_{\alpha,+}(\widehat\Omega_{s})
\le C\int_{s}^{S}\delta(\tau)\,\dd\tau.
\end{equation}
The corrected entropy
\begin{equation}\label{eq:corrected-entropy}
        \mathcal{E}_{\alpha,+}(\hat{\Omega}_s) +C\int_s^\infty\delta(\tau)\,\dd\tau
\end{equation}
is therefore nonincreasing after enlarging $C$, and the entropy $\cE_{\alpha,+}(\widehat{\Omega}_s)$ has a finite limit as $s\to\infty$.
\end{proposition}

To avoid differentiating the supremum in the definition of the entropy, we fix a terminal time and transport its entropy point backward. This standard device is used in \cite{GN17,AGN2016,CH2024} and ensures that the chosen centre remains admissible at all earlier times.

Now fix a terminal time $S$ and $z_e(S)$ such that $\mathcal{E}_{\alpha,+}(\hat{\Omega}_S)= \mathcal{E}_{\alpha,+}(\hat{\Omega}_S,z_e(S))$ and $\rho(S)z_e(S)$ belongs to the unnormalized half-body at time $S$.  For $S-1 \leq s\le S$ set
\begin{equation}\label{eq:transported-support}
        u^{S}(\xi,s)
        =h(\xi,s)-e^{s-S}\langle z_e(S),\xi\rangle.
\end{equation}
The function $\rho(s)u^{S}(\cdot,s)$ is the
support function of the unnormalized half-body at time $s$ with respect to the fixed point $\rho(S)z_e(S)$. By \eqref{equ-uS}, $u^{S}>0$ on every earlier time slice under consideration. Moreover, it satisfies
\begin{equation*}
        \partial_su^{S}=u^{S}-\dfrac{\Theta K^{\alpha}}{\dfrac{1}{|\mathbb{S}_+^n|}\displaystyle\int_{\mathbb{S}_+^n}\Theta K^{\alpha-1}\mathrm{d}\sigma},
        \qquad \overline{\nabla}_\eta u^{S}=0\quad \text{on}~\partial\mathbb{S}^n_+.
\end{equation*}
For $\alpha\ne1$, differentiating the entropy at this transported center gives
\begin{equation}\label{eq:transported-entropy-derivative}
\frac{\dd}{\dd s}
\cE_{\alpha,+}(\widehat\Omega_s,e^{s-S}z_e(S)) = 1-
\frac{\dfrac{1}{|\mathbb{S}_+^n|}\displaystyle\int_{\mathbb{S}_+^n}\Theta K^\alpha (u^{S})^{-\frac{1}{\alpha}}\mathrm{d}\sigma}
     {\dfrac{1}{|\mathbb{S}_+^n|}\displaystyle\int_{\mathbb{S}_+^n}\Theta K^{\alpha-1}\mathrm{d}\sigma
      \dfrac{1}{|\mathbb{S}_+^n|}\displaystyle\int_{\mathbb{S}_+^n}(u^{S})^{1-\frac{1}{\alpha}}\mathrm{d}\sigma}.
\end{equation}
For $\alpha=1$, the corresponding identity is
\begin{equation}\label{eq:transported-entropy-derivative-one}
\frac{\dd}{\dd s}
\mathcal{E}_{1,+}(\widehat\Omega_s,e^{s-S}z_e(S))
=1-\frac{\displaystyle\int_{\mathbb{S}_+^n}\Theta K/(u^{S})\mathrm{d}\sigma}{\displaystyle\int_{\mathbb{S}_+^n}\Theta \mathrm{d}\sigma}.
\end{equation}


\begin{lemma}
\label{lem:transported-entropy-alpha}
Let $\alpha > \frac{1}{n+2}$.  There are $s_1<\infty$ and $C<\infty$ such that, whenever $S\ge s\ge s_1$ with $S-s\leq 1$ and $u^{S}$ is defined by
\eqref{eq:transported-support},
\begin{equation}\label{eq:entropy-almost-monotone}
\frac{\dd}{\dd s}
\mathcal{E}_{\alpha,+}(\widehat\Omega_s,e^{s-S}z_e(S))
\le C\delta(s).
\end{equation}
\end{lemma}

\begin{proof}
Set $J:=\det(\Id+\rho q B)$.
Then
\begin{equation*}
\Theta = \Lambda^{n\alpha+1} J^\alpha,\qquad \mbox{where}\ \Lambda = |E+\rho Y|^2.
\end{equation*}
Lemma \ref{lem:delta-decay-alpha} and \eqref{eq:delta-def} give
\begin{equation*}
        |\Lambda-1|+|\Lambda^{n\alpha+1}-1|\le C\delta(s).
\end{equation*}
In particular, there exists a constant $C$, depending only on $n,\alpha$ and the initial data, such that, for all sufficiently large $s$,
\begin{equation}
    0 < (1 - C \delta(s)) J^\alpha \leq \Theta \leq (1 + C \delta(s)) J^\alpha,
    \label{equ-Theta-bound}
\end{equation}
and
\begin{equation}
    \frac{1 - C\delta(s)}{1 + C\delta(s)} \geq 1 - 2C\delta(s)>0.
    \label{equ-factor-est}
\end{equation} 
The conformal convexity proved in Lemma \ref{lem:large-s-gauss-alpha} gives
$A_Z>0$.  Since the origin belongs to the half-body, $h\ge0$. Hence
\eqref{s5.Dq} gives $q\le0$. Therefore the eigenvalues of $\rho q B$ are
$-\lambda_i$ with $0\le\lambda_i<1$.  Hence
\begin{equation}\label{eq:J-basic}
        0<J=\prod_i(1-\lambda_i)\le1,
        \qquad
        1-J\le\sum_i\lambda_i\le C\rho\tr B,
\end{equation}
since $|q|\leq C$ for sufficiently large $s$.

The divergence theorem implies that, for a constant vector $z\in\Pi$,
\begin{align*}
\int_{\mathbb{S}_+^n}\dfrac{\langle z,\xi \rangle}{K}\mathrm{d}\sigma=&\int_{Y(M,s)} \langle z,N \rangle\mathrm{d}\mu+\int_{\widehat{\Omega}_s\cap\Pi}\langle z,\mu \rangle \mathrm{d}s \\
=& \int_{\widehat{\Omega}_s}\mathrm{div}~z~\mathrm{d}x=0.
\end{align*}
Equation \eqref{eq:half-minkowski-normalized} now gives
\begin{equation*}
	\dfrac{1}{|\mathbb{S}_+^n|}\int_{\mathbb{S}_+^n} \frac{u^S}{K} \mathrm{d}\sigma
	=\dfrac{1}{|\mathbb{S}_+^n|}\int_{\mathbb{S}_+^n} \frac{h}{K}\mathrm{d}\sigma=1.
\end{equation*}
Set $g=\dfrac{J^\alpha K^\alpha}{u^S}>0$.  With respect to the probability measure $\mathrm{d}v=\dfrac{u^S}{K|\mathbb{S}^n_+|}d\sigma$,
\begin{align*}
        \int g\,\mathrm{d}v
        &=\frac{1}{|\mathbb S_+^n|}
          \int_{\mathbb S_+^n}J^\alpha K^{\alpha-1}\,\mathrm{d}\sigma,\\
        \int g^{\frac{1}{\alpha}}\,\mathrm{d}v
        &=\frac{1}{|\mathbb S_+^n|}
          \int_{\mathbb S_+^n}J (u^S)^{1-\frac{1}{\alpha}}\,\mathrm{d}\sigma,\\
        \int g^{1+\frac{1}{\alpha}}\,\mathrm{d}v
        &=\frac{1}{|\mathbb S_+^n|}
          \int_{\mathbb S_+^n}J^{\alpha+1}K^\alpha
          (u^{S})^{-\frac1\alpha}\,\mathrm{d}\sigma.
\end{align*}
H\"older's inequality gives
\begin{equation*}
        \int g\,\mathrm{d}v
        \leq \left(\int g^{1+\frac{1}{\alpha}}\,\mathrm{d}v \right)^{\frac{\alpha}{\alpha+1}},
        \qquad
        \int g^{\frac{1}{\alpha}}\,\mathrm{d}v
        \leq \left(\int g^{1+\frac{1}{\alpha}}\,\mathrm{d}v\right)^{\frac{1}{\alpha+1}}.
\end{equation*}
Multiplying these two inequalities yields
\begin{align}\label{equ-holder}
    \dfrac{1}{|\mathbb{S}_+^n|}\int_{\mathbb{S}_+^n} J^{\alpha+1}K^{\alpha}(u^S)^{-\frac{1}{\alpha}}\mathrm{d}\sigma \geq &~\left( \dfrac{1}{|\mathbb{S}_+^n|}\int_{\mathbb{S}_+^n} J^{\alpha}K^{\alpha-1}\mathrm{d}\sigma \right)\nonumber\\
    &\qquad \times \left( \dfrac{1}{|\mathbb{S}_+^n|}\int_{\mathbb{S}_+^n} J(u^S)^{1-\frac{1}{\alpha}}\mathrm{d}\sigma \right).    
\end{align}
Since $0 < J \leq 1$, we also have $J^\alpha \geq J^{\alpha+1}$, and therefore
\begin{equation}
    \dfrac{1}{|\mathbb{S}_+^n|}\int_{\mathbb{S}_+^n} J^{\alpha}K^{\alpha}(u^S)^{-\frac{1}{\alpha}}\mathrm{d}\sigma \geq \dfrac{1}{|\mathbb{S}_+^n|}\int_{\mathbb{S}_+^n} J^{\alpha+1}K^{\alpha}(u^S)^{-\frac{1}{\alpha}}\mathrm{d}\sigma.
    \label{equ-holder1}
\end{equation}

For $\alpha>\frac{1}{n+2}$, combining \eqref{eq:transported-entropy-derivative}, \eqref{eq:transported-entropy-derivative-one}, \eqref{equ-Theta-bound}, \eqref{equ-factor-est}, \eqref{equ-holder} and \eqref{equ-holder1}, we obtain
\begin{align}
    \frac{\dd}{\dd s}\mathcal{E}_{\alpha,+}(\widehat\Omega_s,e^{s-S}z_e(S)) \leq & 1- \dfrac{\dfrac{1}{|\mathbb{S}_+^n|}\displaystyle\int_{\mathbb{S}_+^n}J^{\alpha}K^{\alpha}(u^S)^{-\frac{1}{\alpha}}\mathrm{d}\sigma}{\dfrac{1}{|\mathbb{S}_+^n|}\displaystyle\int_{\mathbb{S}_+^n}J^{\alpha}K^{\alpha-1}\mathrm{d}\sigma\dfrac{1}{|\mathbb{S}_+^n|}\int_{\mathbb{S}_+^n}(u^S)^{1-\frac{1}{\alpha}}\mathrm{d}\sigma}(1-2C\delta(s)) \nonumber\\
    \leq & 1- \dfrac{\dfrac{1}{|\mathbb{S}_+^n|}\displaystyle\int_{\mathbb{S}_+^n}J^{\alpha+1}K^{\alpha}(u^S)^{-\frac{1}{\alpha}}\mathrm{d}\sigma}{\dfrac{1}{|\mathbb{S}_+^n|}\displaystyle\int_{\mathbb{S}_+^n}J^{\alpha}K^{\alpha-1}\mathrm{d}\sigma\dfrac{1}{|\mathbb{S}_+^n|}\int_{\mathbb{S}_+^n}(u^S)^{1-\frac{1}{\alpha}}\mathrm{d}\sigma}(1-2C\delta(s)) \nonumber\\
    \leq & 1 - \dfrac{\displaystyle\int_{\mathbb{S}_+^n}J (u^S)^{1-\frac{1}{\alpha}}\mathrm{d}\sigma}{\displaystyle\int_{\mathbb{S}_+^n} (u^S)^{1-\frac{1}{\alpha}}\mathrm{d}\sigma} (1-2C\delta(s))   \nonumber\\
    \leq & \dfrac{\displaystyle\int_{\mathbb{S}_+^n} (1-J)(u^S)^{1-\frac{1}{\alpha}}\mathrm{d}\sigma}{\displaystyle\int_{\mathbb{S}_+^n} (u^S)^{1-\frac{1}{\alpha}}\mathrm{d}\sigma} + C\delta(s),  \label{equ-dsentropy}
\end{align}
where we used $0<J\leq 1$ in the last line and the uniform constant $C$ is independent of $s$.

Since $\rho(S) z_e(S)$ belongs to the unnormalized half-body $\Omega_{t(S)}^{\mathbb{H}}\subset \Omega_{t(s)}^{\mathbb{H}}$, we have
\begin{equation}\label{eq:U-upper-delta}
        0<\rho u^S=\rho h- \langle \rho(S) z_e(S), \xi \rangle  \le  C\delta(s).
\end{equation}
Since $\overline{\nabla}_{\eta}\langle z_e(S),\xi \rangle = \langle z_e(S),\eta \rangle=0$ on $\partial\mathbb{S}_+^n$, we have
\begin{equation}\label{eq:rhoB-U}
        B=\overline{\nabla}^2 u^S+ u^S\sigma,
        \qquad \overline{\nabla}_\eta u^S=0\ \mathrm{on}\ \partial\mathbb{S}_+^n.
\end{equation}
Write the eigenvalues of $-\rho qB$ as $\lambda_i\in[0,1)$.  Since $|q|\le C$ for large $s$,
\begin{align*}
        0\le1-J
        =&1-\prod_{i=1}^n(1-\lambda_i)\\
        \le &\sum_{i=1}^n\lambda_i
        \le C\rho\tr B
        =C\rho(\overline{\Delta} u^S+nu^S).
\end{align*}
Using the Neumann condition in \eqref{eq:rhoB-U} and integrating
by parts, we obtain
\begin{align}
        \int_{\mathbb{S}^n_+}(1-J)(u^S)^{1-\frac{1}{\alpha}}\,\dd\sigma
        &\le C\rho\int_{\mathbb{S}^n_+}(\overline{\Delta} u^S+nu^S)(u^S)^{1-\frac{1}{\alpha}}\,\dd\sigma     \nonumber\\
        &=C\rho\int_{\mathbb{S}_+^n} \left( \frac{1-\alpha}{\alpha}\dfrac{|\overline{\nabla} u^S|_{\sigma}^2}{u^S} + n u^S \right)(u^S)^{1-\frac{1}{\alpha}}\dd \sigma\label{eq:BJ-B-direct-alpha}
\end{align}
For $\alpha\geq 1$, we have $\frac{1-\alpha}{\alpha}\leq 0$, so \eqref{eq:U-upper-delta} gives
\begin{align}\label{equ-alphageq1}
	\int_{\mathbb{S}^n_+}(1-J)(u^S)^{1-\frac{1}{\alpha}}\,\dd\sigma  \leq &~ C\rho \int_{\mathbb{S}^n_+} (u^S)^{2-\frac{1}{\alpha}} \,\dd\sigma \nonumber\\
    \leq &~C\delta(s)\int_{\mathbb{S}^n_+} (u^S)^{1-\frac{1}{\alpha}} \,\dd\sigma.
\end{align}
For $\frac{1}{n+2}<\alpha<1$, we have $\frac{1-\alpha}{\alpha}>0$.  Lemma \ref{lem:unit-strip-transport-C0} gives
\begin{equation*}
	  \dfrac{1}{C}\leq u^S \leq C,\quad |\overline{\nabla}u^S|_{\sigma} \leq C.
\end{equation*}
 Substituting this into \eqref{eq:BJ-B-direct-alpha} and noting that $\rho(s)\leq C\delta(s)$, we obtain
\begin{align}
	\int_{\mathbb{S}^n_+}(1-J)(u^S)^{1-\frac{1}{\alpha}}\,\dd\sigma \leq &C\rho(s)\int_{\mathbb{S}^n_+}(u^S)^{1-\frac{1}{\alpha}}\,\dd\sigma \nonumber\\
    \leq &C\delta(s)\int_{\mathbb{S}^n_+}(u^S)^{1-\frac{1}{\alpha}}\,\dd\sigma.
    \label{equ-alpha<1}
\end{align}
Equations \eqref{equ-alphageq1} and \eqref{equ-alpha<1} therefore give, for $\alpha>\frac{1}{n+2}$,
\begin{equation*}
        \dfrac{\displaystyle\int_{\mathbb{S}_+^n} (1-J)(u^S)^{1-\frac{1}{\alpha}}\mathrm{d}\sigma}{\displaystyle\int_{\mathbb{S}_+^n} (u^S)^{1-\frac{1}{\alpha}}\mathrm{d}\sigma}  \leq C\delta(s).
\end{equation*}
Substitution into \eqref{equ-dsentropy} proves \eqref{eq:entropy-almost-monotone}.
\end{proof}

\begin{remark}
The role of Lemma \ref{lem:transported-entropy-alpha} depends on the sign of $1-\frac{1}{\alpha}$.

If $\alpha\ge1$, then $1-\frac{1}{\alpha}\geq 0$.  In
this case, the integration-by-parts identity \eqref{eq:BJ-B-direct-alpha} contains the nonpositive gradient contribution
\begin{equation*}
        -\frac{\alpha-1}{\alpha}\int_{\mathbb{S}^n_+}(u^S)^{-\frac{1}{\alpha}}
        |\overline\nabla u^S|_\sigma^2\,\dd \sigma \le 0.
\end{equation*}
The estimate of
Lemma \ref{lem:transported-entropy-alpha} therefore holds for every admissible support function
\begin{equation*}
        u(\xi)=h(\xi,s)-\langle z,\xi\rangle,\qquad z\in\PiH,
        \qquad u>0\quad\hbox{on }\Sph^n_+ .
\end{equation*}

By contrast, if $\frac1{n+2}<\alpha<1$, then $1-\frac{1}{\alpha}<0$, and the same integration by parts produces the positive weighted gradient term
\begin{equation*}
        -\frac{\alpha-1}{\alpha}\int_{\mathbb{S}^n_+}(u^S)^{-\frac{1}{\alpha}}
        |\overline\nabla u^S|_\sigma^2\,\dd\sigma.
\end{equation*}
This term can be singular for a general admissible translation because the corresponding support function may approach zero.  In this range, Lemma \ref{lem:unit-strip-transport-C0} applies after the entropy point has been chosen at the terminal time $S$ and transported backward. The resulting support function
\begin{equation*}
        u^S(\xi,s)=h(\xi,s)-e^{s-S}\langle z_e(S),\xi\rangle,
        \qquad S-1\le s\le S,
\end{equation*}
satisfies uniform $C^0$ and $C^1$ bounds on the unit strip. These bounds control the gradient term in \eqref{eq:BJ-B-direct-alpha}. 

For $\frac1{n+2}<\alpha<1$, Lemma \ref{lem:transported-entropy-alpha} is therefore used only for the transported functions $u^S$. For $\alpha\ge1$, it is valid for all admissible centers.
\end{remark}

With Lemma \ref{lem:transported-entropy-alpha} in hand, the almost-monotonicity follows by integrating the derivative estimate and comparing the entropy at the earlier time with its value at the transported centre.

\begin{proof}[Proof of Proposition \ref{thm:almost-monotonicity-alpha}]
Recall that $u^S$ is positive on $[s,S]$. At the terminal time $S$,
\begin{equation*}
        \mathcal{E}_{\alpha,+}(\widehat\Omega_{S},z_e(S))
        =\mathcal{E}_{\alpha,+}(\widehat\Omega_{S}),
\end{equation*}
whereas at the earlier time the supremum property gives
\begin{equation*}
        \cE_{\alpha,+}(\widehat\Omega_{s})
        \ge
        \cE_{\alpha,+}(\widehat\Omega_{s},e^{s-S}z_e(S)).
\end{equation*}
Integrating \eqref{eq:entropy-almost-monotone} in $[s,S]$ therefore yields
\begin{align}
    \mathcal{E}_{\alpha,+}(\widehat{\Omega}_S)-\mathcal{E}_{\alpha,+}(\widehat{\Omega}_s) \leq & \mathcal{E}_{\alpha,+}(\widehat{\Omega}_S,z_e(S)) - \mathcal{E}_{\alpha,+}(\widehat{\Omega}_s,\mathrm{e}^{s-S}z_e(S)) \nonumber \\
    =& \int_s^S \dfrac{\mathrm{d}}{\mathrm{d}\tau} \mathcal{E}_{\alpha,+}(\widehat{\Omega}_{\tau},\mathrm{e}^{\tau-S}z_e(S)) ~\mathrm{d}\tau \nonumber \\
    \leq & C \int_s^S \delta(\tau)~\mathrm{d}\tau,  \label{equ-dentropy}
\end{align}
which is \eqref{eq:two-time-entropy-inequality}.  Chaining \eqref{eq:two-time-entropy-inequality} over intervals of length at most one gives the same inequality for arbitrary later times $S\ge s\ge s_1$, and hence proves the monotonicity of \eqref{eq:corrected-entropy}.

Since $\delta\in L^1([s_1,\infty))$, the corrected entropy $\cE_{\alpha,+}(\widehat{\Omega}_s)+C\int_s^{\infty}\delta(\tau)\dd\tau$ is nonincreasing.  Proposition \ref{prop:half-entropy-geometry} bounds $\mathcal{E}_{\alpha,+}(\widehat{\Omega}_s)$ below by zero, so the entropy has a finite limit.
\end{proof}

The almost monotonicity property (Proposition \ref{thm:almost-monotonicity-alpha}), together with the static entropy geometry (Proposition \ref{prop:half-entropy-geometry}), immediately gives uniform radius bounds for the normalized bodies.

\begin{corollary}[Uniform radius bounds]\label{cor:normalized-radius-bounds}
There is $C<\infty$ such that, for all sufficiently large $s$,
\begin{equation*}
        C^{-1}\le r_-(\widehat\Omega_s)
        \le r_+(\widehat\Omega_s)\le C,
\end{equation*}
and the same two-sided bounds hold for the minimum and maximum widths.
\end{corollary}

\begin{proof}
Proposition \ref{thm:almost-monotonicity-alpha} bounds the ordinary entropy from above, and \eqref{eq:entropy-radius-control} gives the asserted radius and width bounds.
\end{proof}

\subsection{The $C^0$ and $C^1$ estimates}
We now apply the almost monotonicity result (Proposition~\ref{thm:almost-monotonicity-alpha}) to show that the entropy point $z_e(s)$ converges to the origin. Once $z_e(s)\to0$, the static entropy geometry gives uniform upper and lower bounds for the support function $h$ as well as upper bounds for $|\overline{\nabla}h|_{\sigma}$.

\begin{proposition}
	\label{thm:C0-alpha}
	Assume $\alpha>\frac{1}{n+2}$. Then
	\begin{equation}
    \label{eq:entropy-point-to-zero}
		z_e(s)\longrightarrow 0.
	\end{equation}
	Consequently, after increasing the initial time, the support function satisfies
	\begin{equation}\label{eq:subunit-origin-C0}
		c\le h(\xi,s)\le C,\qquad |\overline{\nabla}h|_{\sigma}(\xi,s)\leq C,
		\qquad \xi\in\mathbb{S}^n_+.
	\end{equation}
\end{proposition}

\begin{proof}
	The uniform entropy upper bound from Proposition
	\ref{thm:almost-monotonicity-alpha} and the static stability estimate
	\eqref{eq:entropy-quadratic-stability} give a constant $c>0$ such that, for
	all large $s$ and every admissible $z\in\PiH$,
	\begin{equation}\label{eq:strict-gap-use-subunit}
		\cE_{\alpha,+}(\widehat\Omega_s)
		-\cE_{\alpha,+}(\widehat\Omega_s,z)
		\ge c\min\{1,|z-z_e(s)|^2\} .
	\end{equation}
	Let $S\in[s,s+1]$.  Applying \eqref{eq:strict-gap-use-subunit} with $z=e^{s-S}z_e(S)$ and then using \eqref{equ-dentropy} yields
	\begin{equation*}
		\begin{aligned}
			c\min\{1,|z_e(s)-e^{s-S}z_e(S)|^2\}
			&\le \cE_{\alpha,+}(\widehat\Omega_s)
			-\cE_{\alpha,+}(\widehat\Omega_s,e^{s-S}z_e(S)) \\
			&\le \cE_{\alpha,+}(\widehat\Omega_s)
			-\cE_{\alpha,+}(\widehat\Omega_S)
			+C\int_s^S\delta(\tau)\mathrm{d}\tau.
		\end{aligned}
	\end{equation*}
	The right-hand side tends to zero uniformly for $0\le S-s\le1$ by Proposition \ref{thm:almost-monotonicity-alpha} and $\delta\in L^1([s_0,+\infty))$.  Therefore,
	\begin{equation*}
		\lim\limits_{s\to\infty} \sup_{0\le S-s\le1}
		|z_e(s)-e^{s-S}z_e(S)|=0 .
	\end{equation*}
	Taking $S=s+1$ gives
	\begin{equation}\label{eq:ze-unstable-recurrence}
    |z_e(s)-e^{-1}z_e(s+1)|\to 0.
	\end{equation}
    Since $z_e(s)$ is uniformly bounded, taking the upper limit gives
    \begin{equation*}
        \limsup\limits_{s\to\infty}|z_e(s)| \leq \mathrm{e}^{-1} \limsup\limits_{s\to\infty} |z_e(s)|.
    \end{equation*}
    Therefore $\limsup_{s\to\infty}|z_e(s)|=0$, and hence $z_e(s)\to 0$ as $s\to\infty$.
    
	The upper bounds for $h$ and $|Y|$ follow from Corollary \ref{cor:normalized-radius-bounds}. Applying
	Proposition \ref{prop:half-entropy-geometry} to the doubled bodies gives
	constants $d_0>0$ such that
	\begin{equation*}
		\dist\big(z_e(s),\partial(\widehat\Omega_s)^d\big)\ge d_0.
	\end{equation*}
	Together with $z_e(s)\to0$, this yields the lower bound for $h$: for all sufficiently large $s$,
	\begin{align*}
		h(\xi,s)
		=&u_{z_e(s)}(\xi,s)+\langle z_e(s),\xi\rangle \\
		\ge &d_0-|z_e(s)|
		\ge \frac{d_0}{2}.
	\end{align*}
	The bound for $|\overline{\nabla}h|_{\sigma}$ follows from convexity of the hypersurface. This completes the proof of \eqref{eq:subunit-origin-C0}.
\end{proof}

\section{Curvature and higher-order estimates}
\label{sec:apriori-alpha}

Throughout this section, $\alpha>\frac{1}{n+2}$. For
convenience, we first collect the notation used below. All
unmarked geometric quantities refer to the normalized
hypersurface $Y(\cdot,s)$ in the Euclidean half-space. In Gauss
map coordinates,
\begin{equation*}
    B_{ij}
    =
    \overline{\nabla}_i\overline{\nabla}_j h
    +h\sigma_{ij},
    \qquad
    A=B^{-1},
    \qquad
    K=\det A=(\det B)^{-1}.
\end{equation*}
We write
\begin{equation*}
    \rho=e^{-s},
    \qquad
    \omega=\rho q,
\end{equation*}
where $\Lambda$ and $q$ are defined in \eqref{s5.Dq}.
The shifted Weingarten map and its inverse are denoted by
\begin{equation*}
    \widehat A=A+\omega\Id>0,
    \qquad
    \widehat B=\widehat A^{-1},
    \qquad
    \widehat K=\det\widehat A.
\end{equation*}
We also set
\begin{equation*}
    J=\det(\Id+\omega B)
      =\frac{\widehat K}{K},
    \qquad
    P=A\widehat B A.
\end{equation*}
Thus the conformal coefficient and the shifted speed are
\begin{equation*}
    \Theta=\Lambda^{n\alpha+1}J^\alpha,
    \qquad
    W:=\Theta K^\alpha
      =\Lambda^{n\alpha+1}\widehat K^\alpha.
\end{equation*}
With this notation, the normalized support-function equation
takes the form
\begin{equation}\label{eq:F-shifted}
    \partial_s h=h-\frac{W}{\zeta},
\end{equation}
where
\begin{equation*}
    \zeta(s)
    =
    \frac{1}{|\mathbb S^n_+|}
    \int_{\mathbb S^n_+}\frac{W}{K}\,\mathrm{d}\sigma
    =
    \frac{1}{|\mathbb S^n_+|}
    \int_{\mathbb S^n_+}
    \Lambda^{n\alpha+1}
    J\widehat K^{\alpha-1}\,\mathrm{d}\sigma.
\end{equation*}

\subsection{Upper and lower bounds for the shifted speed}

\begin{lemma}[Upper speed bound]\label{prop:F-upper-alpha}
There is $C<\infty$ such that
\begin{equation}\label{eq:F-upper-alpha}
        W(\xi,s)\le C
\end{equation}
for all $s$ sufficiently large.
\end{lemma}

\begin{proof}
    Let $t=t(s)$ and let $\bar{K}$ denote the Gauss curvature of the original hypersurface $M_t\subset\B^{n+1}$.  By Corollary~\ref{cor:normalized-radius-bounds}, for all sufficiently large $s$ the inradius and
circumradius of the unnormalized half-space body are both comparable
to $\rho(s)$. Since the Cayley map and its inverse are uniformly
bi-Lipschitz on the relevant neighborhood of the origin, there exists
a point $p_t\in\Omega_t$ such that
\begin{equation*}
    B_{2c\rho(s)}(p_t)\subset\Omega_t,
    \qquad
    \sup_{M_t}|X-p_t|\le C\rho(s).
\end{equation*}
Since the bodies are nested, the same ball is contained in
$\Omega_\tau$ for every $\tau\le t$. Applying
\eqref{equ-G-upper-time} in Lemma \ref{prop:Tso-alpha} with
$\sigma=c\rho(s)$ and $R_t\le C\rho(s)$ gives
\begin{equation}\label{equ-original-K-upper-alpha}
    \bar K^\alpha(\cdot,t(s))
    \le C\rho(s)^{-n\alpha}.
\end{equation}
    
    The conformal curvature relation is
    \begin{align*}
        \bar{K} = &e^{-nw} \rho(s)^{-n}\det(A+\rho(s)q \mathrm{Id}) \\
        = &e^{-nw}\rho(s)^{-n} \widehat K = 2^{-n}\Lambda^n\rho(s)^{-n}\widehat{K}.
    \end{align*}
    Together with $\Lambda=1+O(\delta(s))$ and \eqref{equ-original-K-upper-alpha}, this identity gives
    \begin{equation}
        \label{equ-wideKupper}
        \widehat{K} \leq C
    \end{equation}
    and
    \begin{equation*}
        W=\Lambda^{n\alpha+1}\widehat K^\alpha \leq C\rho(s)^{n\alpha}\bar{K}^\alpha(x,t) \leq C.
    \end{equation*}
\end{proof}

\begin{lemma}\label{lem:Lambda-bounds-alpha}
There is $C\ge1$ such that
\begin{equation}\label{eq:Lambda-bounds-alpha}
        C^{-1}\le \zeta(s)= \dfrac{1}{|\mathbb{S}_+^n|}\int_{\mathbb{S}_+^n} \Theta K^{\alpha-1}\mathrm{d}\sigma  \le C
\end{equation}
for all sufficiently large $s$.
\end{lemma}

\begin{proof}
We first estimate the average of $J$. By \eqref{eq:J-basic}, $0\leq 1-J\leq C\rho\tr B$.   
Since $\rho$ depends only on time, and 
\begin{equation*}
	\tr B=\overline\Delta h+nh,
	\qquad
	\overline\nabla_\eta h=0
	\quad\text{on }\partial\mathbb S^n_+,
\end{equation*}
integration by parts and the Neumann boundary condition
give
\begin{align*}
	\frac{1}{|\mathbb S^n_+|}
	\int_{\mathbb S^n_+}(1-J)\,d\sigma
	&\leq
	\frac{C}{|\mathbb S^n_+|}
	\int_{\mathbb S^n_+}\rho\tr B\,d\sigma\\
	&=
	\frac{Cn}{|\mathbb S^n_+|}
	\int_{\mathbb S^n_+}\rho h\,d\sigma\\
	&\leq C\rho(s)
	\leq C\delta(s).
\end{align*}
Here we used the uniform upper bound for $h$ from
\eqref{eq:subunit-origin-C0} and the estimate $\rho(s)\leq C\delta(s)$
established above. 

Then for all sufficiently large $s$,
\begin{equation}
    \label{equ-Jintegral}
    \dfrac{1}{2} \leq \dfrac{1}{|\mathbb{S}_+^n|}\int_{\mathbb{S}_+^n} J \mathrm{d}\sigma \leq 1.
\end{equation}
Using \eqref{eq:half-minkowski-normalized} and the identity $K=\widehat K/J$ we also have
\begin{equation}\label{eq:hJ-hatK}
        \dfrac{1}{|\mathbb{S}_+^n|}\int_{\mathbb{S}_+^n}\frac{h J}{\widehat K} \mathrm{d}\sigma = 1.
\end{equation}
Define the probability measure $d\nu_s=\frac{J\,d\sigma}{\int_{\mathbb{S}_+^n} J \mathrm{d}\sigma}$.  Combining \eqref{eq:subunit-origin-C0} with \eqref{eq:hJ-hatK} yields
\begin{equation}\label{eq:negative-moment-hatK-section7}
        c\le\int_{\mathbb{S}^n_+}\widehat K^{-1}\,d\nu_s\le C.
\end{equation}
Now rewrite $\zeta(s)$ as
\begin{equation*}
        \zeta(s)=\dfrac{1}{|\mathbb{S}_+^n|}\int_{\mathbb{S}_+^n}  J \mathrm{d}\sigma \int_{\mathbb{S}^n_+}\Lambda^{n\alpha+1}\widehat K^{\alpha-1}\,d\nu_s.
\end{equation*}
Since $\Lambda=1+O(\delta(s))$ and \eqref{equ-Jintegral}, it suffices to bound the last term from above and below.

If $\alpha=1$, the assertion is immediate.  Suppose that $\alpha>1$.  H\"older's inequality gives
\begin{align*}
    1 =& \int_{\mathbb{S}_+^n} (\widehat{K}^{-1})^{\frac{\alpha-1}{\alpha}} \widehat{K}^{\frac{\alpha-1}{\alpha}} \mathrm{d}\nu_s \\
    \le & \left( \int_{\mathbb{S}_+^n} \widehat{K}^{-1}\,\mathrm{d}\nu_s \right)^{\frac{\alpha-1}{\alpha}} \left( \int_{\mathbb{S}_+^n} \widehat{K}^{\alpha-1}\,\mathrm{d}\nu_s \right)^{\frac{1}{\alpha}}.
\end{align*}
Hence \eqref{eq:negative-moment-hatK-section7} yields
\begin{equation*}
        \int_{\mathbb{S}_+^n}\widehat K^{\alpha-1}\,d\nu_s\ge c .
\end{equation*}
The upper bound follows from \eqref{equ-wideKupper}.  This proves \eqref{eq:Lambda-bounds-alpha} when $\alpha>1$.

Suppose now that $\frac{1}{n+2} < \alpha < 1$.
The upper speed bound $\widehat{K} \le C$ gives
$\widehat{K}^{\alpha-1} \ge C^{\alpha-1}$, which is the desired lower bound.  For the upper bound, the function $x \mapsto x^{1-\alpha}$ is concave on $(0,\infty)$ because $0 < 1-\alpha < 1$.  Jensen's inequality and \eqref{eq:negative-moment-hatK-section7} give
\begin{equation*}
    \int_{\mathbb{S}_+^n} \widehat{K}^{\alpha-1}\,\mathrm{d}\nu_s
    = \int_{\mathbb{S}_+^n} (\widehat{K}^{-1})^{1-\alpha}\,\mathrm{d}\nu_s
    \le \left( \int_{\mathbb{S}_+^n} \widehat{K}^{-1}\,\mathrm{d}\nu_s \right)^{1-\alpha}
    \le C .
\end{equation*}
Thus \eqref{eq:Lambda-bounds-alpha} follows for every $\alpha > \frac{1}{n+2}$.
\end{proof}

\begin{lemma}[Lower speed bound]\label{prop:F-lower-alpha}
There is $c>0$ such that
\begin{equation}\label{eq:F-lower-alpha}
       W(\xi,s)\ge c
\end{equation}
for all $s$ sufficiently large.
\end{lemma}

\begin{proof}
Recall that
\begin{equation*}
    \omega=\rho q,
    \qquad
    \widehat B=\widehat A^{-1},
    \qquad
    P=A\widehat B A.
\end{equation*}
We define the operator
\begin{equation*}
        \mathcal L
        =\partial_s-\frac{\alpha W}{\zeta}P^{ij}\overline{\nabla}_i\overline{\nabla}_j.
\end{equation*}

We first compute $\mathcal{L}W$. Using \eqref{eq:F-shifted} we obtain
\begin{equation*}
        \partial_sB_{ij}
        =\overline{\nabla}_i\overline{\nabla}_j \partial_sh+\partial_sh\sigma_{ij}
        =B_{ij}-\frac{1}{\zeta}(W_{ij}+W\sigma_{ij}).
\end{equation*}
Consequently
\begin{equation*}
        \partial_sA=-A(\partial_sB)A
        =-A+\frac{1}{\zeta} A(\overline{\nabla}^2 W+ W\sigma)A,
\end{equation*}
and therefore
\begin{equation}\label{eq:hatA-evol-lower-alpha}
        \partial_s\widehat A
        =-A+\frac{1}{\zeta} A(\overline{\nabla}^2 W+ W\sigma)A+(\partial_s\omega)\Id .
\end{equation}
Differentiating
\begin{equation*}
        \log W=(n\alpha+1)\log \Lambda+\alpha\log\widehat K
\end{equation*}
and using \eqref{eq:hatA-evol-lower-alpha}, we obtain the exact identity
\begin{align}
        \frac{\mathcal L W}{W}
        &=
        (n\alpha+1)\,\partial_s\log \Lambda
        -\alpha\tr(\widehat B A)
        +\frac{\alpha W}{\zeta}\tr P
        +\alpha(\partial_s\omega)\tr\widehat B   \nonumber \\
        &=
        -n\alpha
        +(n\alpha+1)\,\partial_s\log \Lambda
        +\alpha(\partial_s\omega+\omega)\tr\widehat B
        +\frac{\alpha W}{\zeta}\tr P,   \label{eq:LF-rough}
\end{align}
where in the last equality we used
\begin{equation*}
        A=\widehat A-\omega\Id,\qquad
        \tr(\widehat B A)=n-\omega\tr\widehat B .
\end{equation*}
Similarly, using $h_{ij}=B_{ij}-h\sigma_{ij}$ we obtain
\begin{align}
        \mathcal{L} h
        &=h-\frac{W}{\zeta}
          -\frac{\alpha W}{\zeta}P^{ij}h_{ij} \nonumber  \\
        &=h-\frac{W}{\zeta}
          -\frac{\alpha W}{\zeta}P^{ij}B_{ij}
          +\frac{\alpha W}{\zeta}h\tr P     \nonumber     \\
        &=h-\frac{W}{\zeta}
          -\frac{\alpha W}{\zeta}\tr(\widehat B A)
          +\frac{\alpha W}{\zeta}h\tr P.  \label{eq:Lh-exact-lower-alpha}
\end{align}

Fix $\ell>0$ and consider the function $\log(Wh^\ell)$. At its spatial minimum,  we have
\begin{equation}\label{eq:min-condition-Fh-alpha}
        \frac{W_i}{W}+\ell\frac{h_i}{h}=0 .
\end{equation}
Using \eqref{eq:LF-rough} and \eqref{eq:Lh-exact-lower-alpha} we obtain at such a point
\begin{align}
       \mathcal L\log(Wh^\ell)
        &=
        \frac{\mathcal L W}{W}
        +\ell\frac{\mathcal L h}{h}
        +\frac{\alpha W}{\zeta}P^{ij}
          \left(
             \frac{W_iW_j}{W^2}
             +\ell\frac{h_i h_j}{h^2}
          \right)                        \nonumber              \\
        &=
        \ell-n\alpha
        +(n\alpha+1)\,\partial_s\log \Lambda
        +\alpha(\partial_s\omega+\omega)\tr\widehat B    \nonumber   \\
        &\quad
        -\ell\frac{W}{\zeta h}
        -\ell\frac{\alpha W}{\zeta h}\tr(\widehat B A)
        +(\ell+1)\frac{\alpha W}{\zeta}\tr P        \nonumber      \\
        &\quad
        +\frac{\alpha W}{\zeta}P^{ij}
          \left(
             \frac{W_iW_j}{W^2}
             +\ell\frac{h_i h_j}{h^2}
          \right).           \label{eq:LlogFh-exact-alpha}
\end{align}
The last line is nonnegative.  The term involving $\tr P$ is also
nonnegative and will be discarded from below.

It remains to estimate the terms coming from $\Lambda$ and $\omega=\rho q$.
At the minimum point of $\log(Wh^\ell)$, \eqref{eq:min-condition-Fh-alpha} gives
\begin{equation*}
        \overline{\nabla} \partial_s h
        =\overline{\nabla} h-\frac{\overline{\nabla} W}{\zeta}
        =\left(1+\ell\frac{W}{\zeta h}\right)\overline{\nabla} h.
\end{equation*}
Since $h$ and $\overline{\nabla} h$ are uniformly bounded by the $C^0$ estimate and convexity, we have
\begin{equation}
    |\partial_s h|+|\overline{\nabla}\partial_s h|_{\sigma}\leq C\left( 1+\dfrac{W}{\zeta h} \right).
    \label{s6.partialshbd}
\end{equation}
Recalling the definitions
\begin{align*}
        \Lambda=&1+2\rho \langle \overline{\nabla}h+h\xi,E \rangle+\rho^2(h^2+|\overline{\nabla}h|_{\sigma}^2), \\
        q=&-\frac{2(\langle \xi,E \rangle+\rho h)}{\Lambda},\qquad \rho_s=-\rho,
\end{align*}
and using the $C^0$ and $C^1$ estimates, \eqref{s6.partialshbd}, and $\Lambda=1+O(\delta(s))$ at the same point, we get
\begin{align}
    |\partial_s\log \Lambda| 
    \le C\rho\left(1+\frac{W}{\zeta h}\right).  \label{eq-Lambdabd}
\end{align}
Therefore, combining the $C^0$ and $C^1$ estimates with $\Lambda=1+O(\delta(s))$, \eqref{s6.partialshbd}, and \eqref{eq-Lambdabd}, yields
\begin{align*}
        |\,\partial_s\omega+\omega\,|
        =&\rho|\partial_s q|\\
        =&  \rho\left| \dfrac{2\rho}{\Lambda}(h-\partial_s h)-q\partial_s\log\Lambda \right|\\
        \le & C \rho^2 \left(1+\frac{W}{\zeta h}\right).
\end{align*}
The conformal relation $\widehat A=\rho e^w\overline A$ and the lower curvature bound \eqref{eq:kappa-lower-alpha} for $\overline{A}$ give
\begin{equation*}
        \widehat A\ge c\rho\Id , 
        \qquad
        \rho\tr\widehat B\le C.
\end{equation*}
Also,
\begin{equation}
        \tr(\widehat B A)
        =n-\omega\tr\widehat B,
        \qquad
        |\omega|\tr\widehat B\le C.
        \label{s6.BA}
\end{equation}
Combining \eqref{eq-Lambdabd}-\eqref{s6.BA} with \eqref{eq:LlogFh-exact-alpha} gives, at a spatial minimum of $\log(Wh^\ell)$,
\begin{equation}\label{eq:logFhl}
        \mathcal L\log(Wh^\ell)
        \ge
        \ell-n\alpha
        -C_1\rho
        -C_\ell\frac{W}{\zeta h}.
\end{equation}

Choose $\ell=n\alpha+2$ and $s_3$ sufficiently large that $C_1\rho\le 1$.  By Proposition \ref{thm:C0-alpha} and Lemma \ref{lem:Lambda-bounds-alpha}, there is a
constant  $\Lambda_0>1$ such that
\begin{equation*}
        \zeta\ge\Lambda_0^{-1},
        \qquad
        \Lambda_0^{-1}\le h\le\Lambda_0
\end{equation*}
for all $s\ge s_3$. Then \eqref{eq:logFhl} implies
\begin{equation*}
    \mathcal L\log(Wh^{\ell})\ge 1-C_2 W
\end{equation*}
at a spatial minimum point.

Choose $\lambda>0$ so small that
\begin{equation*}
        \lambda < \frac{1}{4C_2\Lambda_0^{\ell}}
\end{equation*}
and, after decreasing $\lambda$ if necessary, $ Wh^\ell(\cdot,s_3)\ge 2\lambda$.  

We claim that
\begin{equation*}
        Wh^\ell\ge\lambda
        \qquad\hbox{on }\Sph^n_+\times[s_3,\infty).
\end{equation*}
Suppose otherwise.  Let $s_*>s_3$ be the first time at which the minimum of
$Wh^\ell$ equals $\lambda$, and let $\xi_*$ be a minimum point.  
We observe that the following maximum principle argument remains
valid if $\xi_*$ lies on the boundary. Indeed, differentiating the
boundary condition $\overline{\nabla}_{\eta}h=0$ with respect to $s$
and using \eqref{eq:F-shifted}  
gives $\overline{\nabla}_{\eta}W=0$ on $\partial\mathbb S^n_+$.
Consequently,
\begin{equation*}
    \overline{\nabla}_{\eta}\log(Wh^\ell)=0
    \qquad\text{on }\partial\mathbb S^n_+.
\end{equation*}
Since the equator is totally geodesic and
$\overline{\nabla}_{\eta}h=0$, we also have
\begin{equation*}
    B_{\alpha\eta}=0
    \qquad\text{on }\partial\mathbb S^n_+.
\end{equation*}
It follows that $A$, $\widehat A$, $\widehat B$, and
$P=A\widehat BA$ are block diagonal with respect to the decomposition
into the tangential and conormal directions. In particular, $P^{\alpha\eta}=0$. 
At a boundary minimum of $\log(Wh^\ell)$, the one-sided second derivative
test gives
\begin{equation*}
    \left(\overline{\nabla}_{\alpha}
    \overline{\nabla}_{\beta}\log(Wh^\ell)\right)\geq0,
    \qquad
    \overline{\nabla}_{\eta}
    \overline{\nabla}_{\eta}\log(Wh^\ell)\geq0.
\end{equation*}
Therefore,
\begin{equation*}
    P^{ij}\overline{\nabla}_i
    \overline{\nabla}_j\log(Wh^\ell)\geq0.
\end{equation*}
Since $s_*$ is the first time at which the minimum of $Wh^\ell$
reaches $\lambda$, we also have
\begin{equation*}
    \partial_s\log(Wh^\ell)(\xi_*,s_*)\leq0.
\end{equation*}
Thus, whether $\xi_*$ is an interior point or a boundary point,
\begin{equation*}
    \mathcal L\log(Wh^\ell)(\xi_*,s_*)\leq0.
\end{equation*}
On the other hand, at $(\xi_*,s_*)$,
\begin{equation*}
    0
    \geq \mathcal L\log(Wh^\ell)
    \geq 1-C_2W
    =1-C_2\frac{\lambda}{h^\ell}
    \geq 1-\frac14
    >0,
\end{equation*}
which is a contradiction. Hence $Wh^\ell\ge\lambda$.  Since $h\le\Lambda_0$,
we obtain
\begin{equation*}
        W\ge \lambda\Lambda_0^{-\ell}=:c>0.
\end{equation*}
This proves \eqref{eq:F-lower-alpha}.
\end{proof}

Equations \eqref{eq:F-shifted},
\eqref{eq:F-upper-alpha}, and \eqref{eq:F-lower-alpha} imply that, for large $s$,
\begin{equation*}
        0<c\le\widehat K\le C.
\end{equation*}

\subsection{The second derivative estimate}

\begin{lemma}[Uniform convexity]\label{thm:C2-alpha}
There exists $C\ge 1$ such that, for all sufficiently large $s$,
\begin{equation}\label{eq:B-two-sided}
        C^{-1}\sigma_{ij}
        \le B_{ij}=\overline{\nabla}_i\overline{\nabla}_jh+h\sigma_{ij}
        \le C\sigma_{ij}.
\end{equation}
\end{lemma}

\begin{proof}
We continue to use the notation
\begin{equation*}
    \omega=\rho q,
    \qquad
    \widehat A=A+\omega\Id,
    \qquad
    \widehat B=\widehat A^{-1},
    \qquad
    P=A\widehat B A
\end{equation*}
and the operator
\begin{equation*}
        \mathcal L
        =\partial_s-\frac{\alpha W}{\zeta}
          P^{ij}\overline{\nabla}_i\overline{\nabla}_j.
\end{equation*}
The preceding estimates give
\begin{equation}\label{eq:C2-input-estimates}
\begin{gathered}
        c\le W\le C,
        \qquad c\le\zeta\le C,
        \qquad c\le\widehat K=\det\widehat A\le C,\\
        \rho B\le C\sigma,
        \qquad \widehat A\ge c\rho\Id,
        \qquad |\omega|\le C\rho.
\end{gathered}
\end{equation}
Here $\rho B\le C\sigma$ is the preliminary Euclidean curvature lower bound in Lemma~\ref{lem:large-s-gauss-alpha}. All constants below are independent of the terminal time used in the maximum principle argument.

Since $Y=\overline{\nabla} h+h\xi$ and $Y_{,i}=B_i{}^j e_j$, differentiation
of $\Lambda=1+2\rho\langle Y,E\rangle+\rho^2|Y|^2$ gives
\begin{equation}\label{eq:C2-Lambda-first}
        \Lambda_i=2\rho B_i{}^j(\langle e_j,E \rangle + \rho h_j).
\end{equation}
Moreover,
\begin{equation*}
        \omega=-\frac{2\rho(\langle\xi,E\rangle+\rho h)}{\Lambda}.
\end{equation*}
Using \eqref{eq:C2-Lambda-first} and
$A_i{}^k(\log\Lambda)_k=2\rho (\langle e_i,E\rangle+\rho h_i)/\Lambda$, we obtain the exact identity
\begin{equation}\label{eq:C2-omega-gradient-identity}
        \omega_i=-\widehat A_i{}^k(\log\Lambda)_k.
\end{equation}

To prove the upper bound for $B$, fix $T>s_3$ and let
\begin{equation*}
        \lambda(\xi,s)=\lambda_{\max}(B(\xi,s)).
\end{equation*}
Suppose that the maximum of $\lambda$ on
$\mathbb{S}^n_+\times[s_3,T]$ is attained at $(\xi_0,s_0)$.  If
$s_0=s_3$, its value is bounded by the data at time $s_3$, so assume
$s_0>s_3$. We divide the maximum-principle argument into five steps.

\medskip
\noindent\textbf{Step 1. The boundary maximum.}

Recall that $\eta$ denotes the outward unit conormal of the
equator. Along $\partial\mathbb S^n_+$, we have $\eta=-E$.
Greek indices denote directions tangent to the equator. Since
$h_\eta=0$ and the equator is totally geodesic in $\mathbb S^n$,
the Codazzi identity for $B$ gives
\begin{equation}\label{eq:C2-boundary-B-identities}
        B_{\alpha\eta}=0,
        \qquad
        B_{\alpha\beta,\eta}
        =B_{\alpha\eta,\beta}=0
        \quad\text{on }\partial\mathbb S^n_+.
\end{equation}
Differentiating $h_\eta=0$ in time and using $\partial_sh=h-{W}/{\zeta(s)}$ 
yields
\begin{equation}\label{eq:C2-boundary-Fnu}
        W_\eta=0
        \quad\text{on }\partial\mathbb S^n_+.
\end{equation}

Suppose that the space-time maximum of
$\lambda_{\max}(B)$ is attained at a boundary point
$(\xi_0,s_0)$. At $\xi_0$, diagonalize the tangential block of
$B$. By \eqref{eq:C2-boundary-B-identities}, the splitting
\begin{equation*}
        T_{\xi_0}\mathbb S^n
        =
        T_{\xi_0}\partial\mathbb S^n_+
        \oplus\operatorname{span}\{\eta\}
\end{equation*}
is invariant under $B$. Denote
\begin{equation*}
        b_\eta=B_{\eta\eta},
\end{equation*}
and denote by $\mu_\eta$ and $p_\eta$ the corresponding
eigenvalues of $\widehat A$ and $P$, respectively. From
\begin{equation*}
        W=\Lambda^{n\alpha+1}\widehat K^\alpha
\end{equation*}
and
\begin{equation*}
        (\log\widehat K)_i
        =
        -P^{pq}B_{pq,i}
        +\omega_i\tr\widehat B,
\end{equation*}
equations \eqref{eq:C2-boundary-B-identities} and
\eqref{eq:C2-boundary-Fnu} give
\begin{equation*}
        \alpha p_\eta B_{\eta\eta,\eta}
        =
        (n\alpha+1)(\log\Lambda)_\eta
        +\alpha\omega_\eta\tr\widehat B.
\end{equation*}
On the equator, \eqref{eq:C2-Lambda-first} and
\eqref{eq:C2-omega-gradient-identity} imply
\begin{equation*}
        (\log\Lambda)_\eta
        =
        -\frac{2\rho b_\eta}{\Lambda}<0,
        \qquad
        \omega_\eta
        =
        -\mu_\eta(\log\Lambda)_\eta.
\end{equation*}
Consequently,
\begin{equation}\label{eq:C2-boundary-normal-eigenvalue}
        p_\eta B_{\eta\eta,\eta}
        =
        (\log\Lambda)_\eta
        \left(
            n+\frac1\alpha
            -\mu_\eta\tr\widehat B
        \right).
\end{equation}

We claim that the normal eigenvalue $b_\eta$ cannot be a largest
eigenvalue of $B$ at $(\xi_0,s_0)$. Indeed, if $b_\eta$ were a
largest eigenvalue of $B$, then $\mu_\eta$ would be a smallest
eigenvalue of $\widehat A$. Hence
\begin{equation*}
        \mu_\eta\tr\widehat B
        =
        \sum_{i=1}^n\frac{\mu_\eta}{\mu_i}
        \leq n.
\end{equation*}
It follows from \eqref{eq:C2-boundary-normal-eigenvalue} that
\begin{equation*}
        B_{\eta\eta,\eta}<0.
\end{equation*}

On the other hand, extend $\eta$ by parallel transport along the
inward unit normal geodesic
\begin{equation*}
        \gamma(r)
        =
        \exp_{\xi_0}(-r\eta),
        \qquad r\geq0,
\end{equation*}
and set
\begin{equation*}
        \phi(r)
        =
        B_{\gamma(r),s_0}
        \big(\eta(r),\eta(r)\big).
\end{equation*}
If $b_\eta$ were a largest eigenvalue, then
\begin{equation*}
        \phi(r)
        \leq
        \lambda_{\max}
        \big(B(\gamma(r),s_0)\big)
        \leq
        \lambda_{\max}
        \big(B(\xi_0,s_0)\big)
        =
        \phi(0)
\end{equation*}
for all sufficiently small $r\geq0$. Therefore $\phi'(0)\leq0$.
Since $\gamma'(0)=-\eta$, this gives
\begin{equation*}
        B_{\eta\eta,\eta}(\xi_0,s_0)\geq0,
\end{equation*}
which is a contradiction. Thus the maximal eigenspace of $B$ at
a boundary maximum is contained in
$T_{\xi_0}\partial\mathbb S^n_+$.

\medskip
\noindent\textbf{Step 2. The spectral maximum principle.}

We now carry out the spectral maximum principle at
$(\xi_0,s_0)$. Choose an orthonormal frame in which
\begin{equation*}
        B_{ij}=b_i\delta_{ij},
        \qquad
        b_1=\cdots=b_m>b_{m+1}\geq\cdots\geq b_n,
        \qquad
        b_1=\lambda(\xi_0,s_0).
\end{equation*}
The indices $a,b\leq m$ refer to the maximal eigenspace, and the
indices $i,j>m$ refer to its orthogonal complement. If
$\xi_0\in\partial\mathbb S^n_+$, we also choose $ e_n=\eta$. 
By \textbf{Step 1}, the vectors $e_1,\ldots,e_m$ are tangent to the
equator.

Set
\begin{equation*}
        \kappa_i=b_i^{-1},
        \qquad
        \mu_i=\kappa_i+\omega,
        \qquad
        r_i=\frac{\mu_i}{\kappa_i}=1+\omega b_i,
        \qquad
        p_i=\frac{\kappa_i^2}{\mu_i}.
\end{equation*}
Thus $\mu_i$ and $p_i$ are the eigenvalues of $\widehat A$ and
$P$, respectively. Since $\omega\leq0$ and $\widehat A>0$, we
have
\begin{equation*}
        0<r_i\leq1,
        \qquad
        \mu_1\leq\mu_i
        \quad\text{for every }i.
\end{equation*}

We next justify the first and second derivative inequalities at
both interior and boundary maximum points. Let
\begin{equation*}
        \mathcal E
        =
        \operatorname{span}\{e_1,\ldots,e_m\}
\end{equation*}
be the maximal eigenspace. For an interior maximum, the usual
eigenvalue perturbation argument may be applied in every spatial
direction. At a boundary maximum, it may be applied in directions
tangent to the equator, since the equator is totally geodesic. In
the conormal direction, we use the inward geodesic
$\gamma(r)=\exp_{\xi_0}(-r\eta)$ for $r\geq0$.

For every tangential direction, the function
$\lambda_{\max}(B)$ has a two-sided maximum at $\xi_0$. Hence the
first variation of the restriction of $B$ to $\mathcal E$
vanishes. In the conormal direction, the same conclusion follows
directly from
\eqref{eq:C2-boundary-B-identities}, since
$\mathcal E\subset T_{\xi_0}\partial\mathbb S^n_+$. We therefore
have
\begin{equation}\label{eq:C2-top-block-first-derivative}
        B_{ab,k}=0
        \qquad
        (a,b\leq m,\ 1\leq k\leq n).
\end{equation}

For completeness, we also give the corresponding second-order
argument. Fix a spatial direction $e_k$. At a boundary point,
when $e_k=\eta$, the following calculation is understood along
the inward geodesic with initial velocity $-\eta$. Parallel
transport the frame along the corresponding geodesic. For any
unit vector
\begin{equation*}
        v=\sum_{a=1}^m v^ae_a\in\mathcal E,
\end{equation*}
consider, after normalization, the vector field
\begin{equation*}
        v(r)
        =
        \sum_{a=1}^m v^ae_a(r)
        +
        r\sum_{q>m}
        \frac{
            \sum_{a=1}^m v^aB_{aq,k}
        }{b_1-b_q}
        e_q(r).
\end{equation*}
Since $b_1$ is the space-time maximum of the largest eigenvalue,
\begin{equation*}
        B_{\gamma(r),s_0}\big(v(r),v(r)\big)
        \leq b_1
\end{equation*}
for sufficiently small $r$. In the conormal direction this holds
for $r\geq0$, and the first derivative at $r=0$ vanishes by
\eqref{eq:C2-boundary-B-identities}. Differentiating twice at
$r=0$ gives
\begin{equation*}
        \sum_{a,b\leq m}v^av^b
        \left(
            B_{ab,kk}
            +
            2\sum_{q>m}
            \frac{B_{aq,k}B_{bq,k}}{b_1-b_q}
        \right)
        \leq0.
\end{equation*}
This is the same eigenvalue perturbation calculation as in
\cite[Lemma~5]{BCD2017}. Since $v\in\mathcal E$ is arbitrary, the
quadratic form in parentheses is nonpositive on $\mathcal E$.
Taking $v=e_1$, we obtain
\begin{equation}\label{eq:C2-spectral-second-variation}
        B_{11,kk}
        +
        2\sum_{q>m}
        \frac{B_{1q,k}^2}{b_1-b_q}
        \leq0
        \qquad
        (1\leq k\leq n).
\end{equation}
At a boundary point, the inequality for $k=n$ is understood in
the one-sided inward sense.

Extend $e_1$ locally as a unit vector field and choose the time
extension so that
\begin{equation*}
        \partial_s e_1=0
        \quad\text{at }(\xi_0,s_0).
\end{equation*}
Then
\begin{equation*}
        B(e_1,e_1)
        \leq
        \lambda_{\max}(B)
        \leq
        \lambda(\xi_0,s_0)
\end{equation*}
on the preceding space-time cylinder, with equality at
$(\xi_0,s_0)$. Since $s_0>s_3$, the one-sided time derivative
satisfies
\begin{equation*}
        \partial_sB_{11}(\xi_0,s_0)\geq0.
\end{equation*}
Since $p_k>0$, equation
\eqref{eq:C2-spectral-second-variation} gives
\begin{equation}\label{eq:C2-spectral-max-inequality}
\begin{split}
        0\leq{}&
        \partial_sB_{11}
        -
        \frac{\alpha W}{\zeta}
        \sum_{k=1}^n p_k
        \left(
            B_{11,kk}
            +
            2\sum_{q>m}
            \frac{B_{1q,k}^2}{b_1-b_q}
        \right).
\end{split}
\end{equation}

\medskip
\noindent\textbf{Step 3. The evolution inequality.}

We next derive the differential inequality satisfied by the largest eigenvalue. Set
\begin{equation*}
        \Psi=\log W=(n\alpha+1)\log\Lambda+\alpha\log\widehat K,
        \qquad
        C_{ij}=B_{ij,1}.
\end{equation*}
From \eqref{eq:C2-top-block-first-derivative}, and also directly from Codazzi, we have
\begin{equation}\label{eq:C2-C-vanishing}
\begin{aligned}
       &  C_{ab}=0\quad(a,b\le m),\\
       & C_{1k}=B_{1k,1}=B_{11,k}=0\quad(1\le k\le n).   
\end{aligned}
\end{equation}
The evolution equation for $B$ is
\begin{equation*}
        \partial_sB_{ij}=B_{ij}-\frac1\zeta(W_{ij}+W\sigma_{ij}),
\end{equation*}
so
\begin{equation*}
        \partial_sB_{11}=b_1-\frac1\zeta(W_{11}+W),
        \qquad
        W_{11}=W(\Psi_{11}+\Psi_1^2).
\end{equation*}
Since $A_{,1}=-ACA$, differentiation of
$\widehat A=A+\omega\Id$ gives
\begin{equation*}
        \widehat A_{,1}=-ACA+\omega_1\Id,
\end{equation*}
\begin{equation*}
        \widehat A_{,11}
        =-AB_{,11}A+2A C A C A+\omega_{11}\Id.
\end{equation*}
Consequently,
\begin{equation*}
        (\log\widehat K)_{11}
        =-P^{pq}B_{pq,11}+\mathcal Q+\mathcal R,
\end{equation*}
where
\begin{equation}\label{eq:C2-QR-definition}
\begin{split}
        \mathcal Q
        &=2\tr(\widehat B A C A C A)
          -\tr(\widehat B A C A\widehat B A C A),\\
        \mathcal R
        &=\omega_{11}\tr\widehat B
          +2\omega_1\tr(\widehat B A C A\widehat B)
          -\omega_1^2\tr(\widehat B^2).
\end{split}
\end{equation}
Using the spherical commutation identity
\begin{equation*}
        B_{ii,11}=B_{11,ii}+B_{ii}-B_{11},
\end{equation*}
we obtain at $(\xi_0,s_0)$
\begin{equation}\label{eq:C2-LB11}
\begin{split}
        \mathcal L B_{11}
        &=b_1-\frac W\zeta
          -\frac{\alpha W}{\zeta}
             \bigl(b_1\tr P-\tr(PB)\bigr)\\
        &\quad
          -\frac W\zeta\left[
              (n\alpha+1)(\log \Lambda)_{11}
              +\Psi_1^2+\alpha\mathcal Q+\alpha\mathcal R
          \right].
\end{split}
\end{equation}
By Codazzi, $B_{1r,k}=B_{rk,1}=C_{rk}$.  Define
\begin{equation*}
        \mathcal{S}
        :=2\sum_{k=1}^n\sum_{r>m}
             \frac{p_k C_{rk}^2}{b_1-b_r}.
\end{equation*}
The term $\mathcal S$ is the spectral correction arising from the possible multiplicity of the largest eigenvalue. Combining \eqref{eq:C2-spectral-max-inequality} and
\eqref{eq:C2-LB11} gives
\begin{align}
        0\le &b_1-\frac W\zeta
          -\frac{\alpha W}{\zeta}
             \bigl(b_1\tr P-\tr(PB)\bigr)  \nonumber\\
        & -\frac W\zeta\bigg[
              (n\alpha+1)(\log\Lambda)_{11}
              +\Psi_1^2
              +\alpha(\mathcal Q+\mathcal R+\mathcal S)
          \bigg].\label{eq:C2-max-full-inequality}
\end{align}

\medskip
\noindent\textbf{Step 4. The quadratic and shift terms.}

We now estimate the last bracket in the differential inequality
\eqref{eq:C2-max-full-inequality}. Since $\Psi_1^2\geq0$, it remains to control the terms $\mathcal Q$, $\mathcal R$, and the spectral correction $\mathcal S$. We first combine
$\mathcal Q$ with $\mathcal S$ and show that their sum controls the potentially negative quadratic terms involving the first derivatives of $B$. We then use the special structure of the shift $\omega=\rho q$ to estimate $\mathcal R$.

Since $h>0$ and $\rho>0$, formula \eqref{s5.Dq} gives
$q<0$ and hence $\omega<0$ at every finite normalized time. Positivity of $\widehat A$ gives
$b_1<-\omega^{-1}$.  Therefore, for every $i>m$,
\begin{equation*}
        b_1-b_i\le -\omega^{-1}-b_i=-\frac{r_i}{\omega},
\end{equation*}
and hence
\begin{equation}\label{eq:C2-key-coefficient}
        \frac1{p_i(b_1-b_i)}
        =\frac{b_i r_i}{b_1-b_i}
        \ge -\omega b_i=1-r_i.
\end{equation}
Denote $r_*=r_1=\cdots=r_m>0$.  Using
\eqref{eq:C2-C-vanishing}, a direct component expansion gives
\begin{equation*}
\begin{split}
        \mathcal Q 
        &= 2\tr(\widehat{A}BPCPC) - \tr(PCPC) \\
        &=\sum\limits_{i,j=1}^n \left(2p_ip_jr_iC_{ij}^2-p_ip_j C_{ij}^2 \right)\\
        &=\sum_{i>m}p_i^2(2r_i-1)C_{ii}^2
          +2\sum_{m<i<j\le n}
             p_i p_j(r_i+r_j-1)C_{ij}^2\\
        &\quad
          +2\sum_{\substack{a\le m\\ i>m}}
             p_a p_i(r_*+r_i-1)C_{ai}^2,
\end{split}
\end{equation*}
and
\begin{equation*}
\begin{split}
        \mathcal S
        &=\sum_{i>m}\frac{2p_i}{b_1-b_i}C_{ii}^2\\
        &\quad
          +2\sum_{m<i<j\le n}
            \left(\frac{p_j}{b_1-b_i}
                  +\frac{p_i}{b_1-b_j}\right)C_{ij}^2\\
        &\quad
          +2\sum_{\substack{a\le m\\ i>m}}
             \frac{p_a}{b_1-b_i}C_{ai}^2.
\end{split}
\end{equation*}
For the diagonal lower-eigenspace terms,
\eqref{eq:C2-key-coefficient} gives
\begin{equation*}
        p_i^2(2r_i-1)+\frac{2p_i}{b_1-b_i}
        \ge p_i^2.
\end{equation*}
For $m<i<j\le n$, it gives
\begin{equation*}
\begin{split}
        &2p_i p_j(r_i+r_j-1)
        +2\left(\frac{p_j}{b_1-b_i}
                 +\frac{p_i}{b_1-b_j}\right)
        \ge2p_i p_j,
\end{split}
\end{equation*}
and for $a\le m$ and $i>m$,
\begin{equation*}
        2p_a p_i(r_*+r_i-1)
        +\frac{2p_a}{b_1-b_i}
        \ge2r_*p_a p_i\ge0.
\end{equation*}
Consequently,
\begin{equation}\label{eq:C2-QE-coercive}
\begin{split}
        \mathcal Q+\mathcal S
        &\ge \sum_{i>m}p_i^2C_{ii}^2
          +2\sum_{m<i<j\le n}p_i p_jC_{ij}^2\\
        &\quad
          +2r_*\sum_{\substack{a\le m\\ i>m}}p_a p_iC_{ai}^2\\
        &\ge \sum_{i>m}p_i^2C_{ii}^2.
\end{split}
\end{equation}

We now control the shift terms in $\mathcal R$.  At the maximum point,
\eqref{eq:C2-Lambda-first} and \eqref{eq:C2-C-vanishing} yield
\begin{equation*}
\begin{split}
        \Lambda_1&=2\rho b_1(\langle e_1,E\rangle+\rho h_1),\\
        \Lambda_{11}
        &=2\rho b_1\bigl(-\langle\xi,E\rangle
                          +\rho(b_1-h)\bigr).
\end{split}
\end{equation*}
The preliminary estimate $\rho b_1\le C$, together with the $C^0$ and
$C^1$ estimates and $\Lambda=1+O(\delta(s))$, gives
\begin{equation}\label{eq:C2-ell-bounds}
        |(\log\Lambda)_1|+|(\log\Lambda)_{11}|\le C.
\end{equation}
By \eqref{eq:C2-omega-gradient-identity}, $ \omega_1=-\mu_1(\log\Lambda)_1$.  
Differentiating \eqref{eq:C2-omega-gradient-identity} in the $e_1$
direction and using
\begin{equation*}
        (\widehat A_1{}^k)_{,1}
        =(-ACA+\omega_1\Id)_1{}^k
        =\omega_1\delta_1^k
\end{equation*}
because $C_{1k}=0$, we obtain
\begin{equation}\label{eq:C2-omega11}
        \omega_{11}
        =-\omega_1(\log\Lambda)_1-\mu_1(\log\Lambda)_{11}
        =\mu_1\left((\log\Lambda)_1^2-(\log\Lambda)_{11}\right).
\end{equation}
At the diagonal point, \eqref{eq:C2-QR-definition} becomes
\begin{equation*}
        \mathcal R
        =\omega_{11}\sum_i\frac1{\mu_i}
          +2\omega_1\sum_i\frac{\kappa_i^2}{\mu_i^2}C_{ii}
          -\omega_1^2\sum_i\frac1{\mu_i^2}.
\end{equation*}
Since $\mu_1\le\mu_i$, equations
\eqref{eq:C2-ell-bounds}--\eqref{eq:C2-omega11} imply
\begin{equation}\label{eq:C2-R-zero-order-bound}
        \left|\omega_{11}\sum_i\frac1{\mu_i}\right|
        +\omega_1^2\sum_i\frac1{\mu_i^2}
        \le C.
\end{equation}
Furthermore, $\kappa_i^2/\mu_i^2=p_i/\mu_i$, and
$C_{aa}=0$ for $a\le m$.  Hence
\begin{equation}\label{eq:C2-R-linear-bound}
\begin{split}
        \left|2\omega_1\sum_i
                    \frac{\kappa_i^2}{\mu_i^2}C_{ii}\right|
        &\le C\sum_{i>m}|p_iC_{ii}|,
\end{split}
\end{equation}
where we used
$|\omega_1|/\mu_i=(\mu_1/\mu_i)|(\log\Lambda)_1|\le C$.
Young's inequality, \eqref{eq:C2-QE-coercive}, and
\eqref{eq:C2-R-zero-order-bound}--\eqref{eq:C2-R-linear-bound} give
\begin{equation}\label{eq:C2-QER-lower}
        \mathcal Q+\mathcal S+\mathcal R
        \ge \frac12\sum_{i>m}p_i^2C_{ii}^2-C
        \ge-C.
\end{equation}
Since $\Psi_1^2\ge0$ and $(\log\Lambda)_{11}$ is bounded, the estimate \eqref{eq:C2-QER-lower} shows that the last bracket in
\eqref{eq:C2-max-full-inequality} is bounded from below by $-C$.

\medskip
\noindent\textbf{Step 5. Completion of the estimate.}

We now return to the main differential inequality \eqref{eq:C2-max-full-inequality}. The estimates in
Step 4 control all terms involving the first derivatives of $B$ and
the shift $\omega$. It remains to estimate the lower-order term
$\tr(PB)$ and to obtain a coercive lower bound for $\tr P$ in terms
of the largest eigenvalue $b_1$.

First, we have
\begin{equation}\label{eq:C2-trPB}
\begin{aligned}
        \tr(PB)
        =&\sum_i\frac{\kappa_i}{\mu_i}\\
          =&\sum_i\left(1-\frac{\omega}{\mu_i}\right)
          =n-\omega\tr\widehat B
          \le C,
\end{aligned}
\end{equation}
where we used 
$\widehat A\ge c\rho\Id$ and $|\omega|\le C\rho$.
Using \eqref{eq:C2-input-estimates}, \eqref{eq:C2-QER-lower}, and
\eqref{eq:C2-trPB} in \eqref{eq:C2-max-full-inequality}, we obtain
\begin{equation}\label{eq:C2-max-reduced}
        0\le b_1-cb_1\tr P+C.
\end{equation}

We next derive a lower bound for $\tr P$.  Since
\begin{equation*}
        \mu_1=\kappa_1+\omega\le\kappa_1=b_1^{-1}
\end{equation*}
and $\prod_i\mu_i=\widehat K\ge c$, we obtain
\begin{equation*}
        \prod_{i=2}^n\mu_i\ge\frac c{\mu_1}\ge cb_1.
\end{equation*}
Moreover, $\kappa_i\ge\mu_i>0$, so
$p_i=\kappa_i^2/\mu_i\ge\mu_i$.  Since $n\ge2$,
\begin{equation}\label{eq:C2-trP-lower}
\begin{split}
        \tr P
        &\ge\sum_{i=2}^n\mu_i
         \ge(n-1)\left(\prod_{i=2}^n\mu_i\right)^{1/(n-1)}\\
        &\ge c b_1^{1/(n-1)}.
\end{split}
\end{equation}
Substituting \eqref{eq:C2-trP-lower} into \eqref{eq:C2-max-reduced}, we get
\begin{equation*}
        0\le b_1-cb_1^{n/(n-1)}+C,
\end{equation*}
which implies $b_1\le C$.  Since $T>s_3$ was arbitrary,
\begin{equation}\label{eq:C2-B-upper}
        B\le C\sigma
        \qquad\text{on }\mathbb{S}^n_+\times[s_3,\infty).
\end{equation}

Finally, \eqref{eq:C2-B-upper} is equivalent to $A\ge c\Id$.  Since
$\omega\to0$ uniformly, after increasing $s_3$ we have
\begin{equation*}
        \widehat A=A+\omega\Id\ge\frac c2\Id.
\end{equation*}
Together with $\det\widehat A\le C$, this implies
$\widehat A\le C\Id$.  Since $\omega\le0$,
\begin{equation*}
        A=\widehat A-\omega\Id
          =\widehat A+|\omega|\Id\le C\Id.
\end{equation*}
Therefore $c\Id\le A\le C\Id$, which is equivalent to
\eqref{eq:B-two-sided}.
\end{proof}

Uniform convexity yields higher-order estimates for \eqref{eq:support-flow-alpha}.
\begin{corollary}[Higher regularity]\label{cor:higher-regularity-alpha}
For every integer $k\ge0$ there is $C_k<\infty$ such that
\begin{equation}\label{eq:Ck-alpha}
        \|h(\cdot,s)\|_{C^k(\Sph^n_+)}\le C_k
\end{equation}
for all sufficiently large $s$.
Moreover,
\begin{equation}\label{eq:coefficient-Ck-limit}
        \|\Theta-1\|_{C^k}
        +\left\|\dfrac{\Theta }{\dfrac{1}{|\mathbb{S}_+^n|}\displaystyle\int_{\mathbb{S}_+^n}\Theta K^{\alpha-1}\mathrm{d}\sigma}-\dfrac{1 }{\dfrac{1}{|\mathbb{S}_+^n|}\displaystyle\int_{\mathbb{S}_+^n} K^{\alpha-1}\mathrm{d}\sigma}
        \right\|_{C^k}
        \le C_k\rho(s).
\end{equation}
\end{corollary}

\begin{proof}
We first verify the structural conditions needed for the
parabolic regularity estimates. For fixed
$(\xi,s,h,\overline{\nabla}h)$, write the second-order part of
\eqref{eq:support-flow-alpha} as
\begin{equation*}
    \mathcal Q(B)
    =
    -\frac{\Lambda^{n\alpha+1}}{\zeta(s)}
    \det(B^{-1}+\omega\Id)^\alpha,
    \qquad
    \omega=\rho q.
\end{equation*}
The calculation in the proof of Lemma
\ref{prop:F-lower-alpha} gives
\begin{equation*}
    \frac{\partial\mathcal Q}{\partial B_{ij}}
    =
    \frac{\alpha W}{\zeta}P^{ij},
    \qquad
    P=A\widehat B A.
\end{equation*}
By Lemmas \ref{prop:F-upper-alpha},
\ref{lem:Lambda-bounds-alpha}, \ref{prop:F-lower-alpha}, and
\ref{thm:C2-alpha}, after increasing the initial normalized time
if necessary, there are constants $0<\lambda\leq\Lambda_0<\infty$
such that
\begin{equation}\label{eq:uniform-parabolicity-higher}
    \lambda\sigma^{ij}
    \leq
    \frac{\alpha W}{\zeta}P^{ij}
    \leq
    \Lambda_0\sigma^{ij}.
\end{equation}
Thus \eqref{eq:support-flow-alpha} is uniformly parabolic, with
constants independent of the normalized time.

We next justify the concavity of the operator. When $\omega=0$,
the second-order part is
\begin{equation*}
    \mathcal Q_0(B)
    =
    -\frac{\Lambda^{n\alpha+1}}{\zeta(s)}
    \det(B)^{-\alpha}.
\end{equation*}
For every symmetric tensor $H$, direct differentiation gives
\begin{align*}
    D_B^2\mathcal Q_0(B)[H,H]
    ={}&
    -\frac{\alpha\Lambda^{n\alpha+1}
    \det(B)^{-\alpha}}{\zeta(s)}
    \bigg(
        \tr(B^{-1}HB^{-1}H)  
        +\alpha\big(\tr(B^{-1}H)\big)^2
    \bigg).
\end{align*}
The uniform bounds
\begin{equation*}
    c\sigma\leq B\leq C\sigma,
    \qquad
    c\leq\zeta\leq C,
    \qquad
    c\leq\Lambda\leq C
\end{equation*}
therefore imply that $\mathcal Q_0$ is uniformly strictly
concave on the compact set of matrices under consideration.
Moreover,
\begin{equation*}
    |\omega|=|\rho q|\leq C\rho(s)\longrightarrow0.
\end{equation*}
The map
\begin{equation*}
    (B,\omega)
    \longmapsto
    -\det(B^{-1}+\omega\Id)^\alpha
\end{equation*}
is smooth whenever $B>0$ and $B^{-1}+\omega\Id>0$.
Consequently, for all sufficiently large $s$, the operator
$\mathcal Q$ is a uniformly small $C^2$ perturbation of
$\mathcal Q_0$ on this compact matrix set. After increasing the
initial time once more, $\mathcal Q$ is therefore uniformly
concave.

The boundary condition
\begin{equation*}
    \overline{\nabla}_\eta h=0
    \qquad\text{on }\partial\mathbb S^n_+
\end{equation*}
is linear and uniformly oblique. Fix a sufficiently large
$S$ and consider the equation on the parabolic cylinder
\begin{equation*}
    \mathbb S^n_+\times[S-2,S].
\end{equation*}
The estimates above are uniform in $S$. By the boundary regularity theory for fully nonlinear uniformly parabolic equations with oblique boundary conditions 
\cite[Theorem~5.6]{ChatzigeorgiouMilakis},
together with the interior parabolic Evans--Krylov estimate, there
exist $\beta\in(0,1)$ and $C<\infty$ such that
\begin{equation}\label{eq:parabolic-C2beta}
    \|h\|_
    {C^{2+\beta,\,1+\beta/2}
    (\mathbb S^n_+\times[S-1,S])}
    \leq C,
\end{equation}
where $C$ is independent of $S$.

The coefficients of the equation are smooth functions of
$\xi$, $\rho$, $h$, $\overline{\nabla}h$, and $B$. The
normalizing factor $\zeta(s)$ is independent of the spatial
variable. Estimate \eqref{eq:parabolic-C2beta} and
the defining integral for $\zeta$ give the required H\"older
continuity of all coefficients. We may therefore apply the
boundary parabolic Schauder estimates successively to the
spatially differentiated equations. The corresponding time
derivatives are then controlled by the equation. Along the
equator, the boundary compatibility relations follow by
differentiating $\overline{\nabla}_\eta h=0$ and using that the equator is
totally geodesic. Consequently, for every integer $m\geq0$,
\begin{equation*}
    \|h\|_
    {C^{m+2+\beta,\,(m+2+\beta)/2}
    (\mathbb S^n_+\times[S-\frac12,S])}    \leq C_m,
\end{equation*}
where $C_m$ is independent of $S$.  Since $S$ is arbitrary, this proves
\eqref{eq:Ck-alpha} for every $k\geq0$.

It remains to prove \eqref{eq:coefficient-Ck-limit}. Recall that $Y=\overline{\nabla}h+h\xi$ 
and
\begin{equation*}
    \Lambda
    =
    1+2\rho\langle Y,E\rangle+\rho^2|Y|^2,
    \qquad
    q
    =
    -\frac{2(\xi_{n+1}+\rho h)}{\Lambda}.
\end{equation*}
The estimates already proved imply, for every $k\geq0$,
\begin{equation}\label{eq:Lambda-q-Ck}
    \|\Lambda-1\|_{C^k}
    \leq C_k\rho(s),
    \qquad
    \|q\|_{C^k}\leq C_k,
    \qquad
    \|B\|_{C^k}\leq C_k.
\end{equation}
It follows that $J=\det(\Id+\rho qB)$ 
satisfies
\begin{equation*}
    \|J-1\|_{C^k}\leq C_k\rho(s).
\end{equation*}
Since $\Theta=\Lambda^{n\alpha+1}J^\alpha$,  
we conclude from \eqref{eq:Lambda-q-Ck} that
\begin{equation}\label{eq:Theta-Ck}
    \|\Theta-1\|_{C^k}\leq C_k\rho(s).
\end{equation}

Finally, define
\begin{equation*}
    \zeta_0(s)
    =
    \frac{1}{|\mathbb S^n_+|}
    \int_{\mathbb S^n_+}K^{\alpha-1}\dd\sigma.
\end{equation*}
Uniform convexity gives $ c\leq K\leq C$ 
and hence
\begin{equation*}
    c\leq\zeta_0(s)\leq C.
\end{equation*}
Together with \eqref{eq:Theta-Ck}, this gives
\begin{align*}
    |\zeta(s)-\zeta_0(s)|
    &=
    \left|
    \frac{1}{|\mathbb S^n_+|}
    \int_{\mathbb S^n_+}
    (\Theta-1)K^{\alpha-1}\dd\sigma
    \right|~\leq ~C\rho(s).
\end{align*}
Using also the two-sided bound for $\zeta(s)$ from Lemma
\ref{lem:Lambda-bounds-alpha}, we obtain
\begin{align*}
    \left\|
        \frac{\Theta}{\zeta(s)}
        -\frac{1}{\zeta_0(s)}
    \right\|_{C^k}
    &\leq
    \frac{\|\Theta-1\|_{C^k}}{\zeta(s)}
    +
    \left|
        \frac{1}{\zeta(s)}
        -\frac{1}{\zeta_0(s)}
    \right|~\leq ~C_k\rho(s).
\end{align*}
This proves \eqref{eq:coefficient-Ck-limit}.
\end{proof}

\section{Convergence to a hemisphere}
\label{sec:convergence-alpha}

\begin{theorem}\label{thm:normalized-convergence-alpha}
Let $\alpha > \frac{1}{n+2}$.  The volume-normalized conformal flow converges in
$C^\infty(\Sph^n_+)$ to the constant support function
\begin{equation*}
        h_\infty\equiv1.
\end{equation*}
Equivalently, the normalized hypersurfaces converge smoothly to the unit
hemisphere centered at the origin in the tangent half-space.
\end{theorem}

\begin{proof}
Let $s_j\to\infty$ and fix $L>0$.  By Corollary
\ref{cor:higher-regularity-alpha}, after passing to a subsequence the shifted
functions
\begin{equation*}
        h_j(\xi,\tau)=h(\xi,s_j+\tau),
        \qquad -L\le\tau\le L,
\end{equation*}
converge smoothly to a positive uniformly convex limit $h_\infty(\xi,\tau)$.
By \eqref{eq:coefficient-Ck-limit}, the limit satisfies the exact normalized
half-space equation
\begin{equation}\label{eq:limit-normalized-flow}
        \partial_\tau h_\infty
        =h_\infty-
        \frac{K_\infty^\alpha}
        {\dfrac{1}{|\mathbb{S}_+^n|} \displaystyle\int_{\mathbb{S}_+^n}K_\infty^{\alpha-1}\dd \sigma},
        \qquad \overline{\nabla}_\eta h_\infty=0\quad \text{on}~\partial\mathbb{S}^n_+.
\end{equation}

The corrected entropy has a limit, and the error tail tends to zero. Hence the entropy of $h_\infty(\tau)$ is independent of $\tau$.  The exact entropy identity of Andrews--Guan--Ni \cite{AGN2016}, applied to the limit equation (or, equivalently, the moment calculation with $J\equiv1$), shows
that a normalized solution with constant entropy has equality in H\"older.
Consequently
\begin{equation}\label{eq:soliton-equation-half}
        K_\infty^\alpha=\lambda(\tau) h_\infty
\end{equation}
for some function $\lambda(\tau)>0$, which is independent of the spatial variable.  Theorem \ref{thm:C0-alpha} and \eqref{eq:entropy-point-to-zero} ensure that the reference point in \eqref{eq:soliton-equation-half} is the origin.  Integrating \eqref{eq:soliton-equation-half} against $K_\infty^{-1}\dd\sigma$ shows
\begin{equation*}
    \lambda(\tau)=\dfrac{1}{|\mathbb{S}_+^n|} \int_{\mathbb{S}_+^n}K_\infty^{\alpha-1}\dd \sigma
\end{equation*}
 and
\eqref{eq:limit-normalized-flow} gives $\partial_\tau h_\infty=0$. Thus $h_\infty$ and $K_\infty$ are independent of $\tau$, and
hence $\lambda(\tau)$ is also constant in time. We denote this
constant by $\lambda$.

Reflect the stationary half-hypersurface across $\PiH$.  The Neumann boundary condition and the flat free-boundary compatibility identities give a smooth, closed, uniformly convex hypersurface satisfying the same soliton equation $K^\alpha=\lambda h$. The classification theorem of Brendle--Choi--Daskalopoulos \cite{BCD2017} implies that this doubled soliton is a round sphere. Reflection symmetry places its center in $\PiH$, while the entropy point is the origin. Hence the center is the origin. The doubled volume is $|\B^{n+1}|$, so its radius is one. Therefore every subsequential limit is $h\equiv1$.

Since all subsequences have the same limit and the family is precompact in every $C^k$ norm, the full flow converges smoothly to $1$.
\end{proof}

\begin{proof}[Proof of Theorem \ref{thmA}]
Proposition \ref{prop:finite-extinction-alpha} proves finite-time contraction
to a point $p\in\Sph^n$ for every $\alpha>0$.  After rotating $p$ to $-E$, the
Cayley map gives the normalized half-space flow of Section
\ref{sec:normalization-alpha}.  For $\alpha > \frac{1}{n+2}$, Theorem
\ref{thm:normalized-convergence-alpha} proves smooth convergence to the unit
hemisphere.  This is the second assertion of Theorem \ref{thmA}.
\end{proof}

\section*{Acknowledgements}
This work was supported by the National Key Research and Development Program of China (2021YFA1001800), National Natural Science Foundation of China (No.~12531002) and the Fundamental Research Funds for the Central Universities.  The third author was also supported by the China Postdoctoral Science Foundation (Grant No.~2025M783146).

\end{document}